\documentclass[a4paper,12pt]{article} 
 
\usepackage{graphicx}

\usepackage{amssymb,amsmath,amsthm}

\newtheorem{theorem}{Theorem}[section]

\newtheorem{corollary}{Corollary}[section]
\newtheorem{proposition}{Proposition}[section]

\theoremstyle{definition}

\newtheorem{example}{Example}[section]
\newtheorem{remark}{Remark}[section]

\DeclareMathOperator{\cone}{cone}
\DeclareMathOperator{\conv}{conv}
\DeclareMathOperator{\dist}{dist}
\DeclareMathOperator{\interior}{int}
\DeclareMathOperator{\ri}{ri}

\newcommand{\R}{\mathbb{R}}

\begin{document}
	
\title{On the Structure of Higher Order Voronoi Cells}


\author{Juan Enrique Mart\'{\i}nez-Legaz\thanks{Departament d'Economia i d'Hist\`{o}ria Econ\`{o}mica, Universitat Aut\`{o}%
		noma de Barcelona, and BGSMath, Spain,
		JuanEnrique.Martinez.Legaz@uab.cat}, 
		Vera~Roshchina\thanks{School of Mathematics and Statistics, UNSW Sydney, Australia, 	v.roshchina@unsw.edu.au}, 
		Maxim Todorov\thanks{	Department of Physics and Mathematics, UDLAP, Puebla, Mexico. On leave from
				Institute of Mathematics and Informatics, BAS, Sofia, Bulgaria, 
				maxim.todorov@udlap.mx}}


\maketitle

\begin{abstract}
The classic Voronoi cells can be generalized to a higher-order version by
considering the cells of points for which a given $k$-element subset of the
set of sites consists of the $k$ closest sites. We study the structure of
the $k$-order Voronoi cells and illustrate our theoretical findings with a
case study of two-dimensional higher-order Voronoi cells for four points.
\end{abstract}

\section{Introduction}

The classic Voronoi cells partition the Euclidean space into polyhedral
regions that consist of points nearest to one of the sites from a given finite
set. We consider higher order (or multipoint) Voronoi cells that correspond to
the subsets of points nearest to $k$ several sites (see an illustration in
Fig.~\ref{fig:six-points-voronoi}). \begin{figure}[th]
\centering
\includegraphics[width=0.8\textwidth]{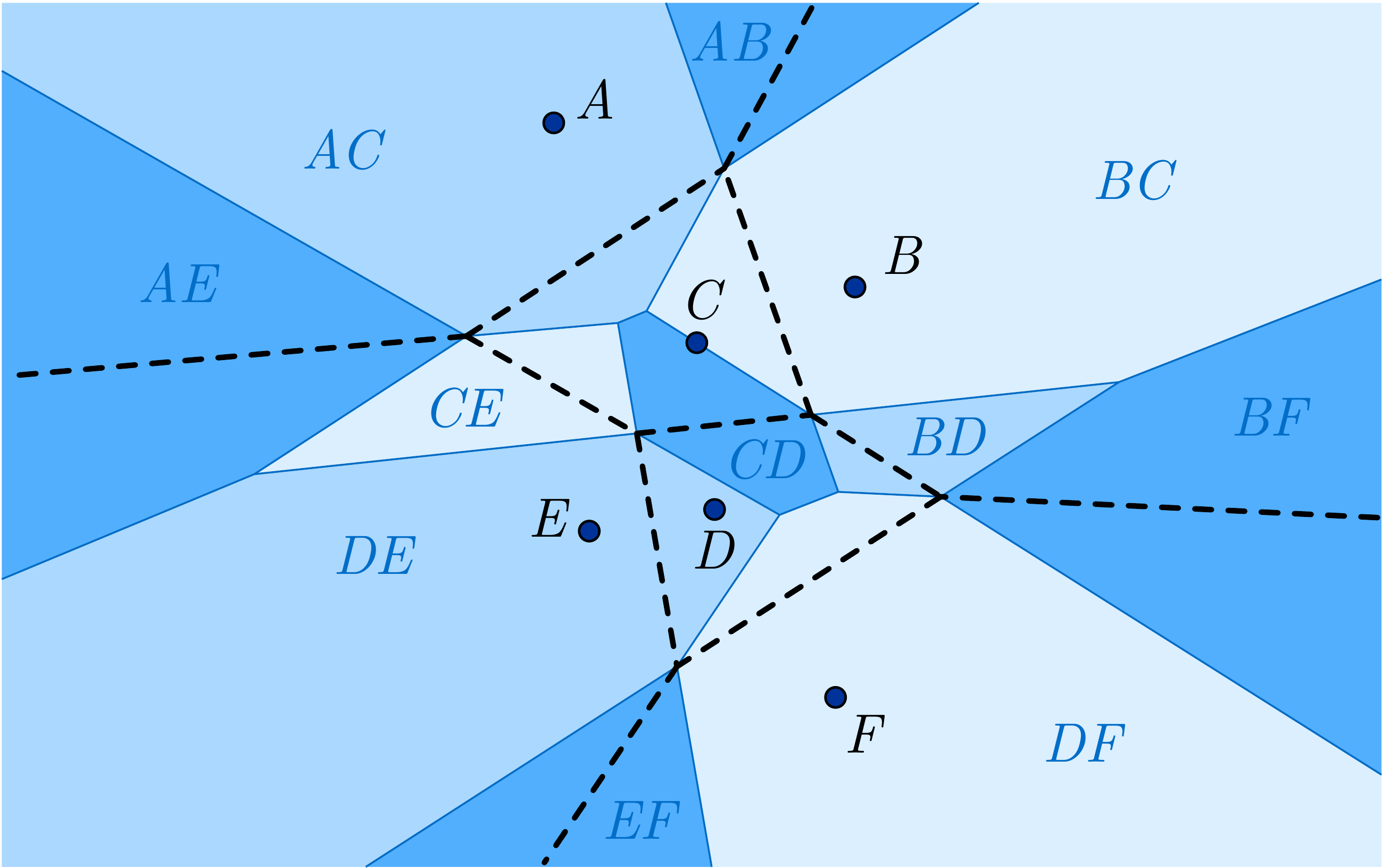}\caption{The
classic and order two Voronoi diagrams on six sites (shown in dashed lines and
shaded regions, respectively). Notice that some two-point combinations
generate empty cells.}%
\label{fig:six-points-voronoi}%
\end{figure}

To our best knowledge, the earliest mention of $k$-point Voronoi cells appears
in \cite{SH75}, where a tessellation of the plane by such cells was called the
\emph{Voronoi diagram of order $k$}; that paper also provides bounds on the
number of nonempty cells in a plane and complexity estimates for the
construction of such diagrams; in \cite{Papadopoulou} the complexity of
constructing the higher-order diagrams for line segments was studied.

The multipoint or $k$-order Voronoi diagrams discussed in this paper are one
possible way to generalize the classic construction. Some notable
generalizations are the cells of more general sets \cite{Papadopoulou,inverse}%
, the use of non-Euclidean metrics \cite{lina,voronoicity,voronoimetrics} and
the abstract cells that are defined via manifold partitions of the space
rather than distance relations \cite{klein}.

Much of the recent work mentioned above is focussed on the algorithmic
complexity of constructing planar Voronoi diagrams of various types. In this
paper we rather focus on the \emph{structure} of multipoint Voronoi cells, and
in particular obtain constructive characterizations of cells with nonempty
interior, also of bounded and empty cells. We use these results for a case
study of multipoint cells defined on at most four sites. We prove
that---perhaps counterintuitively---some convex polygons, including triangles
and cyclic quadrilaterals, can not be such cells, and provide explicit
algorithms for the construction of sites for a given cell in other cases.

Finally, we would like to mention a wealth of emerging application of higher
order Voronoi cells, predominantly driven by the recent advancements in big
data and mobile sensor technology. For instance, in \cite{gradientflow} such
cells are utilized in a numerical technique for smoothing point clouds from
experimental data; in \cite{hole} $k$-order cells are used for detecting and
rectifying coverage problems in wireless sensor networks; in \cite{voting} the
$k$-order diagrams are used to analyze coalitions in the US supreme court
voting decisions. A well-known application of the higher order Voronoi cells
is in a $k$-nearest neighbor problem in spatial networks \cite{spatial},
however, the practical implementations are limited due to the complexity of
higher order diagrams and the lack of readily available software.

Our work is organized as follows. In Section~\ref{sec:theory} we study the
structure of higher order Voronoi cells, and in particular prove the
conditions for the cell to be bounded and nonempty. In
Section~\ref{sec:specialcases} we study the special case of higher order cells
on no more than four sites. We will refer to the higher order cells as
multipoint cells, to highlight the discrete nature of our construction.

\section{High Order Voronoi Cells in $\R^{n}$}

\label{sec:theory}

Let $T\subset\R^{n}$ be finite and nonempty. For a nonempty proper
subset $S\subset T$ we define the multipoint Voronoi cell as the set of points
that are not farther from each point of $S$ than from each point of
$T\setminus S$,
\[
V_{T}(S):=\left\{  x\in\R^{n}\,:\,\max_{s\in S}\dist(s,x)\leq
\min_{t\in T\setminus S}\dist(t,x)\right\}  ,
\]
where $\dist(x,y)$ is the Euclidean distance function. When $S$ is a
singleton, i.e. $S=\{s\}$, the set $V_{T}(S)$ is a classic Voronoi cell. We
abuse the notation slightly and write $V_{T}(s):=V_{T}(\{s\})$. It is evident
that
\begin{equation}
V_{T}(S)=\bigcap_{s\in S}V_{\{s\}\cup\left(  T\setminus S\right)
}(s).\label{eq:evidentprop}%
\end{equation}

It is not difficult to observe that each multipoint Voronoi cell is a convex
polyhedron, i.e. the intersection of finitely many closed halfspaces, since
each cell is defined by finitely many linear inequalities. Explicitly, we have
the following representation.

\begin{proposition}
\label{prop:inequalities}Let $T$ be a finite subset of $\R^{n}$, and
let $S$ be a nonempty and proper subset of $T$. Then $V_{T}\left(  S\right)  $ is
the intersection of $|S|(|T|-|S|)$ closed halfspaces:%
\begin{equation}
V_{T}(S)=\bigcap_{\substack{s\in S \\t\in T\setminus S}}\left\{
x\in\R^{n}\,:\,\langle t-s,x\rangle\leq\frac{1}{2}\left(  \Vert
t\Vert^{2}-\Vert s\Vert^{2}\right)  \right\}  .\label{eq:linrep}%
\end{equation}

\end{proposition}

\begin{proof}
Observe that, from the definition,%
\begin{align*}
V_{T}(S)& =\left\{ x\in \R^{n}\,:\,\max_{s\in S}\dist(x,s)\leq \min_{t\in
T\setminus S}\dist(x,t)\right\} \\
& =\bigcap_{s\in S}\left\{ x\in \R^{n}\,:\,\dist(x,s)\leq \min_{t\in
T\setminus S}\dist(x,t)\right\} \\
& =\bigcap_{s\in S}\bigcap_{t\in T\setminus S}\left\{ x\in \R^{n}\,:\,\dist%
(x,s)\leq \dist(x,t)\right\} .
\end{align*}%
Explicitly for the Euclidean distance function we have
\begin{equation}
\dist(x,s)\leq \dist(x,t)\quad \Leftrightarrow \quad \Vert x\Vert
^{2}-2\langle s,x\rangle +\Vert s\Vert ^{2}\leq \Vert x\Vert ^{2}-2\langle
t,x\rangle +\Vert t\Vert ^{2},  \label{Euclidean distance function}
\end{equation}%
from where the desired representation follows.
\end{proof}

\bigskip

As a consequence of a well known necessary and sufficient condition for the
inconsistency of an arbitrary system of linear inequalities \cite[Theorem
4.4(i)]{GL98}, from (\ref{eq:linrep}) we obtain the following characterization
of empty Voronoi cells.

\begin{theorem}
\label{thm:dualchar}Let $T$ be a finite subset of $\R^{n}$, and let
$S$ be a nonempty and proper subset of $T$. Then%
\[
V_{T}(S)=\emptyset
\]
iff%
\begin{equation}%
\begin{pmatrix}
0_{n}\\
-1
\end{pmatrix}
\in\cone \left\{
\begin{pmatrix}
t-s\\
\Vert t\Vert^{2}-\Vert s\Vert^{2}%
\end{pmatrix}
,\;s\in S,\,t\in T\setminus S\right\}  .\label{eq:dualchar}%
\end{equation}

\end{theorem}

We use the characterization in Theorem~\ref{thm:dualchar} to obtain two well
known statements about the classic Voronoi cells.

\begin{corollary}
\label{nonemptiness}Let $T$ be a finite subset of $\R^{n}$ with
$|T|\geq2.$ Then%
\[
V_{T}(T\setminus\{t\})\neq\emptyset
\]
iff $t$ is an extreme point (vertex) of $\conv T$.
\end{corollary}

\begin{proof}
Suppose that there exists $s\in T$ such that $V_{T}\left( s\right)
=\emptyset .$ Then \eqref{eq:dualchar} holds for $S=\{s\}$, and therefore
there exist $\lambda _{t}\geq 0$ for $t\in T\setminus \{s\}$ such that%
\begin{equation*}
\sum_{t\in T\setminus \{s\}}\lambda _{t}\left( t-s\right) =0_{n}\quad \text{
and }\quad \sum_{t\in T\setminus \{s\}}\lambda _{t}\left( \Vert t\Vert
^{2}-\Vert s\Vert ^{2}\right) =-1,
\end{equation*}%
whereby%
\begin{equation*}
\lambda _{0}:=\sum_{t\in T\setminus \{s\}}\lambda _{t}>0\;\text{ and }%
\;s=\sum_{t\in T\setminus \{s\}}\frac{\lambda _{t}}{\lambda _{0}}t.
\end{equation*}%
Since the the square of the Euclidean norm $\left\Vert .\right\Vert ^{2}$ is
a strictly convex function, we have
\begin{align*}
\sum_{t\in T\setminus \{s\}}\lambda _{t}\left\Vert t\right\Vert ^{2}+1&
=\sum_{r\in T\setminus \{s\}}\lambda _{r}\left\Vert s\right\Vert
^{2}=\lambda _{0}\left\Vert s\right\Vert ^{2}<\lambda _{0}\sum_{t\in
T\setminus \{s\}}\frac{\lambda _{t}}{\lambda _{0}}\Vert t\Vert ^{2} \\
& =\sum_{t\in T\setminus \{s\}}\lambda _{t}\left\Vert t\right\Vert ^{2},
\end{align*}%
which is a contradiction. Now, suppose that $V_{T}(T\setminus
\{t\})=\emptyset $. Denote $S:=T\setminus \{t\}$. From Theorem~\ref%
{thm:dualchar}, there exist $\lambda _{s}\geq 0$ for $s\in S$ such that%
\begin{equation*}
\sum_{s\in S}\lambda _{s}\left( t-s\right) =0_{n}\;\text{ and }\;\sum_{s\in
S}\lambda _{s}\left( \left\Vert t\right\Vert ^{2}-\left\Vert s\right\Vert
^{2}\right) =-1,
\end{equation*}%
whereby%
\begin{equation*}
\sum_{s\in S}\lambda _{s}>0\;\text{ and }\;t=\sum_{s\in S}\frac{\lambda _{s}%
}{\sum_{r\in S}\lambda _{r}}s.
\end{equation*}%
Hence, $t$ is not an extreme point. If $t$ is not an extreme point, then
there exist $\lambda _{s}\geq 0,$ $s\in S$ with $\sum_{s\in S}\lambda _{s}=1$
and $t=\sum_{s\in S}\lambda _{s}s.$ Since $\left\Vert .\right\Vert ^{2}$ is
a strictly convex function,%
\begin{align*}
\left\Vert t\right\Vert ^{2}& <\sum_{s\in S}\lambda _{s}\Vert s\Vert ^{2}; \\
&
\end{align*}%
therefore, setting $\mu _{s}:=\frac{\lambda _{s}}{\sum_{r\in S}\lambda
_{r}\Vert r\Vert ^{2}-\left\Vert t\right\Vert ^{2}}$ for $s\in S$, we have $%
\mu _{s}\geq 0$,
\begin{equation*}
\sum_{s\in S}\mu _{s}\left( t-s\right) =0_{n}\;\text{ and }\;\sum_{s\in
S}\mu _{s}\left( \left\Vert t\right\Vert ^{2}-\left\Vert s\right\Vert
^{2}\right) =-1,
\end{equation*}%
showing that (\ref{eq:dualchar}) holds, which, in view of Theorem \ref%
{thm:dualchar}, proves the corollary.
\end{proof}

\bigskip

The following result generalizes the \textquotedblleft if" statement in the
last part of Corollary \ref{nonemptiness}.

\begin{corollary}
\label{prop:emptycell}Let $T$ be a finite subset of $\R^{n}$, and let
$S$ be a nonempty and proper subset of $T$. If $(\conv S)\cap(T\setminus
S)\neq\emptyset$, then $V_{T}(S)=\emptyset$.
\end{corollary}

\begin{proof}
Taking $t\in(\conv S)\cap(T\setminus S),$ since $t$ is not an extreme point
of $\conv S,$ by Corollary~\ref{nonemptiness} we have $V_{T}(S)\subseteq
V_{S\cup\left\{ t\right\} }(S)=\emptyset.$
\end{proof}

\bigskip

In fact we can prove a more general geometric statement which yields the
preceding corollary.

\begin{theorem}
\label{thm:ballchar}Let $T$ be a finite subset of $\R^{n}$, and let
$S$ be a nonempty and proper subset of $T$. Then%
\[
V_{T}(S)\neq\emptyset
\]
iff there exists a closed Euclidean ball $B\supset S$ such that%
\[
\interior B\cap(T\setminus S)=\emptyset.
\]

\end{theorem}

\begin{proof}
First assume that $V_{T}(S)\neq \emptyset $. Then there exists $c\in
V_{T}(S) $. Let
\begin{equation*}
R:=\max_{s\in S}\dist(c,s),
\end{equation*}%
and let $B$ be the closed Euclidean ball of radius $R$ centered at $c,$
clearly $S\subset B$. Evidently, $\interior B\cap (T\setminus
S)=\emptyset $, otherwise we would have $\dist(c,t)<\dist(c,s)$ for some $%
t\in T\setminus S$ and $s\in S$, hence, $c\notin V_{T}(S)$, which
contradicts our choice of $c$. Now assume that there exists some closed
Euclidean ball $B$ such that $S\subset B$ and $\interior B\cap (T\setminus
S)=\emptyset $. The centre of the ball $B$ is contained in $V_{T}(S)$,
hence, $V_{T}(S)\neq \emptyset $.
\end{proof}

\bigskip

We give an explicit example of an empty cell with $|S|=2$ and $|T|=3$.

\begin{example}
\label{eg:empty} Let $s_{1}=(-1,0)$, $s_{2}=(1,0)$, $t=(0,0)$. It is not
difficult to observe that \eqref{eq:linrep} becomes
\[
V_{T}(S) = \left\{  (x_{1},x_{2})\, : \, x_{1}\leq-\frac{1}{2}, x_{1}\geq
\frac{1}{2}\right\}  = \emptyset.
\]
This configuration is shown in Fig.~\ref{fig:empty-cell-minimal}.
\begin{figure}[th]
\centering \includegraphics[width=0.7\textwidth]{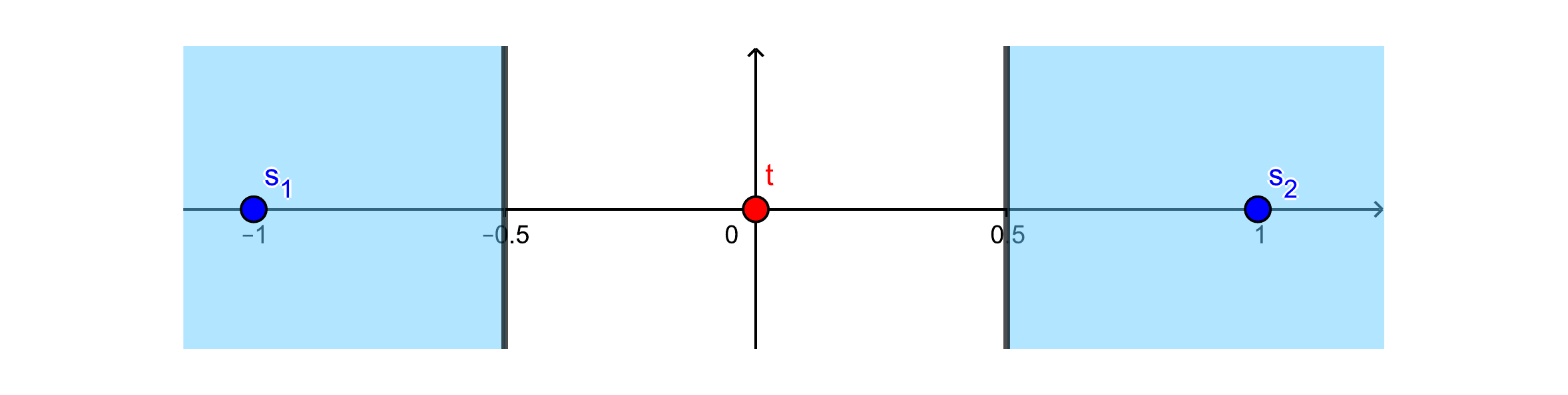}\caption{Minimal
configuration for an empty cell}%
\label{fig:empty-cell-minimal}%
\end{figure}
\end{example}

Notice that we can likewise construct an empty cell with $|T|>|S|\geq2$ by
making sure that $(T\setminus S)\cap\conv S\neq\emptyset$ (using
Corollary~\ref{prop:emptycell}).

\begin{theorem}
\label{thm:boundedcell}Let $T$ be a finite subset of $\R^{n}$, and let
$S$ be a nonempty and proper subset of $T$. Then $V_{T}(S)$ is bounded iff%
\[
\cone \left\{  t-s,s\in S,t\in T\setminus S\right\}  =\R^{n}.
\]

\end{theorem}

\begin{proof}
It suffices to observe that from the linear representation \eqref{eq:linrep}
we obtain that the first moment cone of $V_{T}(S)$ is $cone\left\{
t-s,\;s\in S,t\in T\setminus S\right\} $.
\end{proof}

\begin{remark}
\label{Remark 8}It follows from Theorem~\ref{thm:boundedcell} that if $n\geq2
$ and $|T|\leq3$ all nonempty cells are unbounded.
\end{remark}

We can strengthen the result in the preceding remark as follows.

\begin{theorem}
\label{thm:cardbound}Let $T$ be a finite subset of $\R^{n}$. If
\begin{equation}
|T|<2\sqrt{n+1},\label{eq:cardbound}%
\end{equation}
then for any $S\subset T$ the cell $V_{T}(S)$ is either empty or unbounded.
\end{theorem}

\begin{proof}
Assume that $V_{T}(S)=\emptyset .$ The proof is based on the observation
that a nonempty bounded polyhedron in $\R^{n}$ must be defined by at least $%
n+1$ inequalities. Let $p:=|T|$, $k:=|S|$. Then the number of inequalities
that feature in the representation \eqref{eq:linrep} is $\phi (k)=k(p-k)$.
Observe that $\phi $ attains its maximum at $\frac{p}{2}$ for even $p$ and
at $\frac{p-1}{2}$ for odd $p$. Hence for even $p$
\begin{equation*}
k(p-k)\leq \frac{p}{2}\left( p-\frac{p}{2}\right) =\frac{p^{2}}{4}%
=\left\lfloor \frac{p^{2}}{4}\right\rfloor ,
\end{equation*}%
and for odd $p$
\begin{equation*}
k(p-k)\leq \frac{p-1}{2}\left( p-\frac{p-1}{2}\right) =\frac{p^{2}-1}{4}%
=\left\lfloor \frac{p^{2}}{4}\right\rfloor ,
\end{equation*}%
hence, ensuring \eqref{eq:cardbound} yields at most $n$ inequalities that
define each cell, and so all nonempty cells are unbounded.
\end{proof}

\begin{proposition}
\label{prop:four-three} Let $S\subset T\subset\R^{n}$, with $|S|=3$
and $|T|=4$. Then $V_{T}(S)$ is either empty or unbounded.
\end{proposition}

\begin{proof}
Put $T\setminus S=\left\{ t\right\} .$ In case $t\in \conv S$, then
Corollary~\ref{nonemptiness} gives that $V_{T}(S)=\emptyset $. If $t\notin %
\conv S$, then $t$ can be separated from $S$, and by Theorem~\ref%
{thm:boundedcell} the cell has to be unbounded.
\end{proof}

\bigskip

The following statement will be useful later for a discussion on planar
quadrilateral cells.

\begin{proposition}
\label{prop:intersectfour}Let $S\subset T\subset\R^{2},$ with $S=\{s_1,s_2\}$
and $T\setminus S = \{t_1,t_2\}$. Then
\[
V_{T}(S)\text{ is bounded iff }(s_{1},s_{2})\cap(t_{1},t_{2})\text{ is a
singleton}.
\]

\end{proposition}

\begin{proof}
The configuration of the points of $T$ in Fig.~\ref{fig:four-bounded}
\begin{figure}[th]
\centering \includegraphics[width=0.6\textwidth]{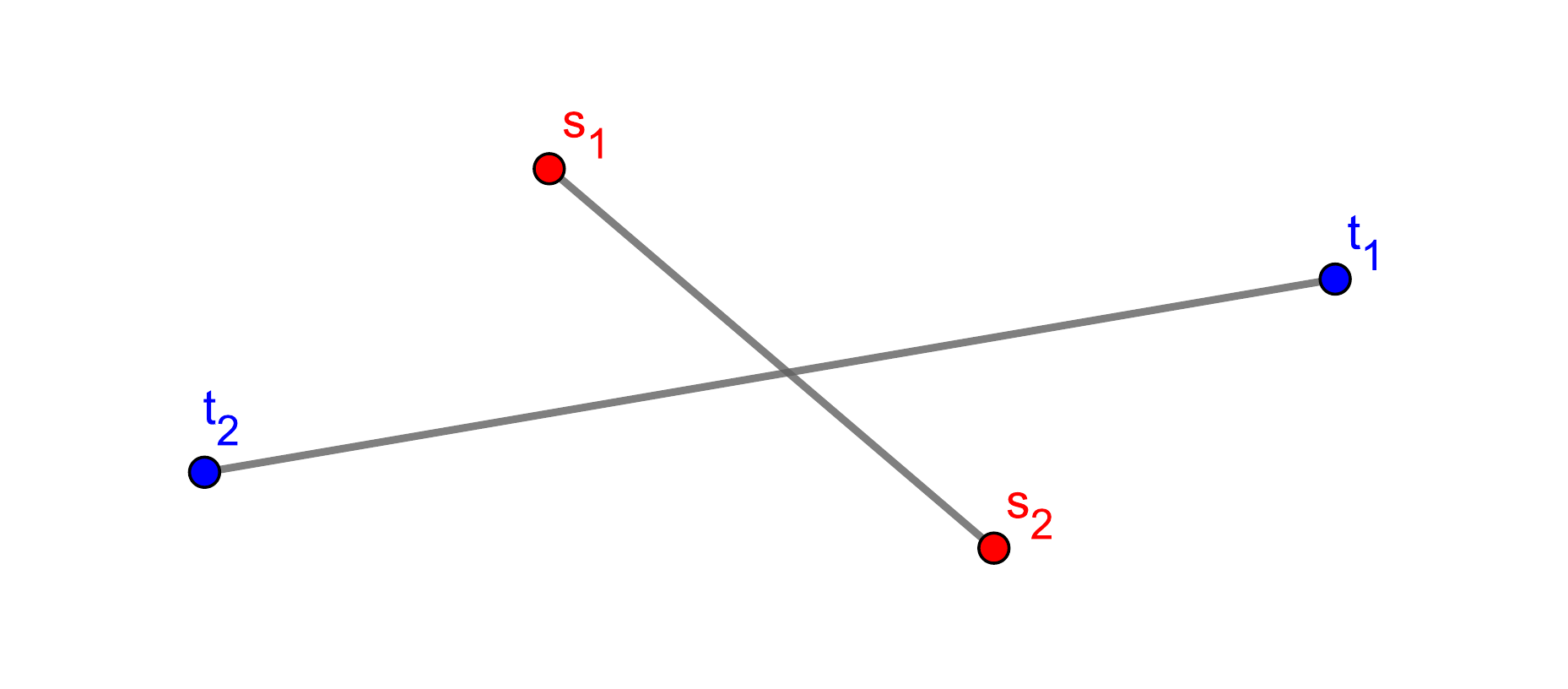}
\caption{Configuration for a bounded cell}
\label{fig:four-bounded}
\end{figure}
means that if we take the line through $s_{1}$ and $s_{2},$ then $t_{1}$ and
$t_{2}$ belong to the two opposite open halfspaces defined by this line. The
same holds true if we interchange $s_{1}$ and $s_{2}$ with $t_{1}$ and $%
t_{2} $. From the linear representation \eqref{eq:linrep} we obtain that the
first moment cone $M$ of $V_{T}(S)$ is equal to $\cone\left\{ t_{i}-s_{j},%
\text{ }i,j=1,2\right\} .$ Let us consider that the configuration of the
points of $T $ is like in Fig.~\ref{fig:four-bounded}. Then, we are going to
prove that $M=%
\R
^{2},$ which, by Theorem \ref{thm:boundedcell}, implies that $V_{T}\left\{
S\right\} $ is bounded$.$ What we are actually going to prove is the
equivalent assertion that the polar cone $M^{\circ }$ reduces to $\left\{
0_{2}\right\} .$ To this aim, let $p\in M^{\circ }$ and assume, w.l.o.g.,
that $(s_{1},s_{2})\cap (t_{1},t_{2})=\left\{ 0_{2}\right\} .$ Then there
exist $\lambda ,\mu >0$ such that $s_{2}=-\lambda s_{1}$ and $t_{2}=-\mu
t_{1}.$ Since $\left\langle p,t_{1}-s_{1}\right\rangle \leq 0$ and $%
\left\langle p,t_{1}+\lambda s_{1}\right\rangle =\left\langle
p,t_{1}-s_{2}\right\rangle \leq 0,$ we have $\left\langle
p,t_{1}\right\rangle \leq 0.$ This inequality combined with $\left\langle
p,-\mu t_{1}-s_{1}\right\rangle =\left\langle p,t_{2}-s_{1}\right\rangle
\leq 0$ yields $\left\langle p,s_{1}\right\rangle \geq 0;$ hence, in view of
$\left\langle p,-\mu t_{1}+\lambda s_{1}\right\rangle =\left\langle
p,t_{2}-s_{2}\right\rangle \leq 0,$ it turns out that $\left\langle
p,t_{1}\right\rangle =0=\left\langle p,s_{1}\right\rangle .$ Since $s_{1}$
and $t_{1}$ are linearly independent because of the assumption $%
(s_{1},s_{2})\cap (t_{1},t_{2})=\left\{ 0_{2}\right\} $, we conclude that $%
p=0,$ as was to be proved. Second, let the Voronoi cell $V_{T}\left\{
s_{1},s_{2}\right\} $ be bounded. Then, by Theorem \ref{thm:boundedcell},
the first moment cone $\cone\left\{ t_{i}-s_{j},\text{ }i,j=1,2\right\} $ is
the whole of $\R^{2}.$ This implies that $t_{1}$ and $t_{2}$ are not
on a common closed halfplane out of the two determined by the straight line through
 $s_{1}$ and $s_{2},$ as otherwise that cone would be contained in the
translate of that half-space with the boundary line passing through the origin, and the same assertion
holds true when we interchange $t_{1}$ and $t_{2}$ with $s_{1}$ and $s_{2}.$
This rules out the possibility that $\conv\left\{
s_{1},s_{2},t_{1},t_{2}\right\} $ be a segment, a triangle, or a
quadrilateral having $s_{1}$ and $s_{2}$ as adjacent vertices. Therefore $%
s_{1}$ and $s_{2}$ are opposite vertices of the quadrilateral $\conv\left\{
s_{1},s_{2},t_{1},t_{2}\right\} ,$ which clearly implies that $%
(s_{1},s_{2})\cap (t_{1},t_{2})$ is a singleton. The proof is completed.
\end{proof}

\begin{theorem}
\label{thm:interior}Let $T$ be a finite subset of $\R^{n}$, and let
$S$ be a nonempty and proper subset of $T$. Then%
\[
\interior V_{T}(S)\neq\emptyset
\]
iff%
\[
0_{n+1}\notin\conv\left\{
\begin{pmatrix}
t-s\\
\left\Vert t\right\Vert ^{2}-\left\Vert s\right\Vert ^{2}%
\end{pmatrix}
,\quad s\in S,\;t\in T\setminus S\right\}  .
\]

\end{theorem}

\begin{proof}
The proof comes from the well known characterization of the Slater condition
for a linear system of inequalities \cite[Theorem 3.1]{GLT96}.
\end{proof}

\begin{example}
\label{eg:singleton} Consider a system $T$ of four points in the plane,
\[
T=\{\left(  0,0\right)  ,\left(  1,1\right)  ,\left(  1,0\right)  ,\left(
0,1\right)  \},\quad S=\{\left(  0,0\right)  ,\left(  1,1\right)  \}.
\]
This is illustrated in Fig.~\ref{fig:singleton-minimal}. \begin{figure}[th]
\centering \includegraphics[width=0.35\textwidth]{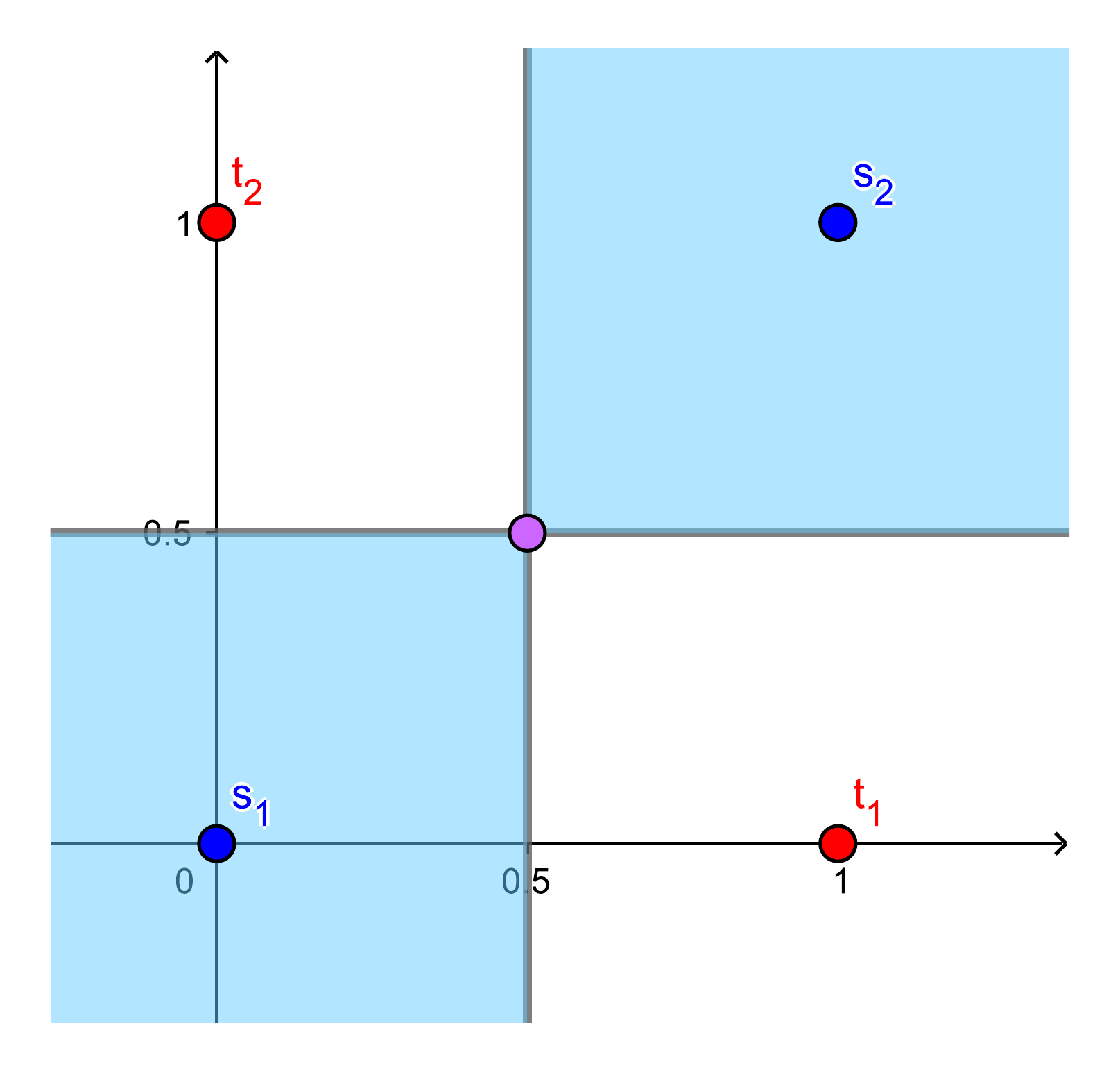}\caption{A
singleton cell (the intersection of the two shaded regions).}%
\label{fig:singleton-minimal}%
\end{figure}Using Theorems~\ref{thm:dualchar},~\ref{thm:boundedcell} and
\ref{thm:interior}, it is easy to check that $V_{T}\left(  S\right)  $ is
nonempty and bounded, but $\interior V_{T}\left(  S\right)  =\emptyset$.
Indeed, $V_{T}\left(  S\right)  =\left\{  \left(  \frac{1}{2},\frac{1}%
{2}\right)  \right\}  $.
\end{example}

The next statement is a specific characterization for a three-point system,
which we will use in what follows.

\begin{proposition}
Let $T,T^{\prime}\subset\R^{n}$ be such that $|T|=|T^{\prime}|=3$ and
$T$ and $T^{\prime}$ differ by exactly one point (i.e. $|T\cap T^{\prime}%
|=2$). Let $s\in T\setminus T^{\prime}$ and $s^{\prime}\in T^{\prime}\setminus
T$. If
\[
V_{T}(s)\subseteq V_{T^{\prime}}(s^{\prime}),
\]
then all points in the set $T^{\prime}\cup T$ belong to the same straight line.
\end{proposition}

\begin{proof}
For notational convenience we will prove the result for $T=\{t_{1},t_{2},0\}$
and $T^{\prime }=\{t_{1},t_{2},s\}$, where the points $t_{1},t_{2},s$ are
all nonzero and pairwise distinct. Let
\begin{equation*}
F:=V_{T}(0)=\left\{ x\in \R^{n}\;:\;\langle t_{j},x\rangle \leq \frac{1}{2}%
\left\Vert t_{j}\right\Vert ^{2},\;j=1,2\right\} ,
\end{equation*}%
\begin{equation*}
F^{\prime }:=V_{T^{\prime }}(s)=\left\{ x\in \R^{n}\;:\;\langle
t_{j}-s,x\rangle \leq \frac{1}{2}\left( \left\Vert t_{j}\right\Vert
^{2}-\Vert s\Vert ^{2}\right) ,\;j=1,2\right\} .
\end{equation*}%
The two inequalities defining $F^{\prime }$ are consequence relations of the
inequalities defining $F$. Therefore there exist $\lambda _{ij}\geq 0,$ $%
i,j=1,2$ such that
\begin{equation}
t_{i}-s=\sum_{j=1}^{2}\lambda _{ij}t_{j},\text{ }i=1,2\text{ and}
\label{eq:obs}
\end{equation}%
\begin{equation*}
\left\Vert t_{i}\right\Vert ^{2}-\left\Vert s\right\Vert ^{2}\geq
\sum_{j=1}^{2}\lambda _{ij}\left\Vert t_{j}\right\Vert ^{2},\text{ }i=1,2.
\end{equation*}%
Hence
\begin{equation}
s=\left( 1-\lambda _{ii}\right) t_{i}-\lambda _{i(3-i)}t_{3-i},\text{ }i=1,2%
\text{ and}  \label{eq:reprs}
\end{equation}%
\begin{equation}
\left\Vert s\right\Vert ^{2}\leq \left( 1-\lambda _{ii}\right) \left\Vert
t_{i}\right\Vert ^{2}-\lambda _{i(3-i)}\left\Vert t_{3-i}\right\Vert ^{2},%
\text{ }i=1,2.\text{ }  \label{eq:reprns}
\end{equation}%
We can subtract the two representations \eqref{eq:reprs} of $s$ to obtain
\begin{equation*}
\left( 1-\lambda _{11}+\lambda _{21}\right) t_{1}+\left( -1-\lambda
_{12}+\lambda _{22}\right) t_{2}=0_{n}.
\end{equation*}%
If $t_{j},$ $j=1,2$ are linearly independent, we have
\begin{equation*}
1=\lambda _{11}-\lambda _{21}\text{ and }1=\lambda _{22}-\lambda _{12}.
\end{equation*}%
Together with \eqref{eq:reprns} this yields
\begin{equation*}
0\leq \left\Vert s\right\Vert ^{2}\leq \left( 1-\lambda _{11}\right)
\left\Vert t_{1}\right\Vert ^{2}-\lambda _{12}\left\Vert t_{2}\right\Vert
^{2}=-\lambda _{21}\left\Vert t_{1}\right\Vert ^{2}-\lambda _{12}\left\Vert
t_{2}\right\Vert ^{2}\leq 0,
\end{equation*}%
which contradicts the condition $s\neq 0$. Therefore $t_{j},$ $j=1,2$ are
linearly dependent. Together with \eqref{eq:obs} this finishes the proof.
\end{proof}

\begin{remark}
\label{rem:enlarge} This proposition means that it is impossible to enlarge a
Voronoi cell of a single point in a three-point affinely independent system by
moving this point.
\end{remark}

\begin{proposition}
\label{prop:union}Let $T$ be a finite subset of $\R^{n}$, and let $S$
be a nonempty proper subset of $T$. If $\left\vert S\right\vert \geq2,$ then
\begin{equation}
V_{T}(S)=\bigcup\limits_{s\in S}\left[  V_{T\setminus\{s\}}(S\setminus
\{s\})\cap V_{S}(s)\right]  .\label{eq:intersect}%
\end{equation}

\end{proposition}

\begin{proof}
Denote
\begin{equation*}
A_{s}:=V_{T\setminus\{s\}}(S\setminus\{s\})\cap V_{S}(s).
\end{equation*}
Observe that for any $\bar{s}\in S$ we have
\begin{equation*}
A_{\bar{s}}=\{x\in\R^{n}\,:\,\dist (x,\bar{s})\leq \dist (x,s)\leq \dist %
(x,t)\quad\forall s\in S,\;t\in T\setminus S\}.
\end{equation*}
Evidently, $A_{s}\subseteq V_{T}(S)$ for every $s\in S$, hence, $%
\bigcup\limits_{s\in S}A_{s}\subseteq V_{T}(S)$. To prove the reverse
inclusion, assume that $x\in V_{T}(S)$. Let $\bar{s}$ be a closest point to $%
x$ in $S$. It is evident that $x\in V_{S}(\bar{s})$. At the same time, it is
not difficult to observe that $V_{T}(S)\subseteq V_{T\setminus{\{\bar{s}\}}%
}(S\setminus\{\bar{s}\})$. Hence
\begin{equation*}
x\in V_{T\setminus{\{\bar{s}\}}}(S\setminus\{\bar{s}\})\cap V_{S}(\bar {s}%
)=A_{\bar{s}},
\end{equation*}
and therefore $V_{T}(S)\subseteq\bigcup\limits_{s\in S}A_{s}$.
\end{proof}


\begin{proposition}
\label{prop:noflats}Let $T$ be a finite subset of $\R^{n}$, and let
$S$ be a nonempty and proper subset of $T$. If there exist $s_{1},s_{2}\in S$ and
$t_{1},t_{2}\in T\setminus S$ such that the inequalities
\begin{equation}
\Vert s_{1}-x\Vert\leq\Vert t_{1}-x\Vert\quad\text{and}\quad\Vert s_{2}%
-x\Vert\leq\Vert t_{2}-x\Vert\label{eq:twoineq}%
\end{equation}
define the same halfspace, then these inequalities are nonessential for
$V_{T}(S)$, i.e. they can be dropped from the system \eqref{eq:linrep}.
\end{proposition}

\begin{proof}
It is evident from the equivalence (\ref{Euclidean distance function}) that
these inequalities can be written as
\begin{equation}
\langle t_{1}-s_{1},x\rangle \leq \frac{1}{2}\left( \Vert t_{1}\Vert
^{2}-\Vert s_{1}\Vert ^{2}\right) ,\quad \langle t_{2}-s_{2},x\rangle \leq
\frac{1}{2}\left( \Vert t_{2}\Vert ^{2}-\Vert s_{2}\Vert ^{2}\right) .
\label{eq:ineqrewritten}
\end{equation}%
If both inequalities in \eqref{eq:twoineq} define the same halfspace, then
it follows from \eqref{eq:ineqrewritten} that there exists $\alpha >0$ such
that
\begin{equation}
t_{2}-s_{2}=\alpha (t_{1}-s_{1})\quad \text{and}\quad \frac{1}{2}\left(
\Vert t_{2}\Vert ^{2}-\Vert s_{2}\Vert ^{2}\right) =\alpha \frac{1}{2}\left(
\Vert t_{1}\Vert ^{2}-\Vert s_{1}\Vert ^{2}\right) .  \label{eq:lindep}
\end{equation}%
Every point $x\in V_{T}(S)$ also satisfies the system
\begin{equation}
\langle t_{1}-s_{2},x\rangle \leq \frac{1}{2}\left( \Vert t_{1}\Vert
^{2}-\Vert s_{2}\Vert ^{2}\right) ,\quad \langle t_{2}-s_{1},x\rangle \leq
\frac{1}{2}\left( \Vert t_{2}\Vert ^{2}-\Vert s_{1}\Vert ^{2}\right) .
\label{eq:essential}
\end{equation}%
Adding these inequalities together we obtain
\begin{equation*}
\langle t_{1}-s_{2}+t_{2}-s_{1},x\rangle \leq \frac{1}{2}\left( \Vert
t_{1}\Vert ^{2}-\Vert s_{2}\Vert ^{2}+\Vert t_{2}\Vert ^{2}-\Vert s_{1}\Vert
^{2}\right) ,
\end{equation*}%
and using \eqref{eq:lindep} we have a consequence of \eqref{eq:essential},
\begin{equation*}
(1+\alpha )\langle t_{1}-s_{1},x\rangle \leq (1+\alpha )\frac{1}{2}\left(
\Vert t_{1}\Vert ^{2}-\Vert s_{1}\Vert ^{2}\right) ,
\end{equation*}%
which defines the same halfspace as \eqref{eq:ineqrewritten}.
\end{proof}

\begin{theorem}
\label{thm:facets}Let $T\subset\R^{n}$ be a finite set, let
$S:=\{s_{1},s_{2}\}\subset T$ be a two-point set$,$ and let%
\begin{align*}
H  &  :=\{x\in\R^{n}\,:\,\Vert x-s_{1}\Vert=\Vert x-s_{2}\Vert\}\\
&  =\{x\in\R^{n}\,:\,\langle s_{1}-s_{2},x\rangle=\frac{1}{2}(\Vert
s_{1}\Vert^{2}-\Vert s_{2}\Vert^{2})\}.
\end{align*}
If $\interior V_{T}(S)\neq\emptyset$, then $H\cap\ri F=\emptyset$ for every
facet $F$ of $V_{T}(S)$.
\end{theorem}

\begin{proof}
Let $F$ be the facet of $V_{T}(S)$ defined by the linear equation
corresponding to some point $t_{0}\in T$ and, say, $s_{1},$ that is, $%
\langle t_{0}-s_{1},x\rangle =\frac{1}{2}\left( \Vert t_{0}\Vert ^{2}-\Vert
s_{1}\Vert ^{2}\right) ,$ and assume, towards a cotradiction, that $H\cap
\ri
F\neq \emptyset .$ Take $\overline{x}\in H\cap \ri F.$ We then have
\begin{equation*}
\langle t_{0}-s_{1},\overline{x}\rangle =\frac{1}{2}(\Vert t_{0}\Vert
^{2}-\Vert s_{1}\Vert ^{2}).
\end{equation*}%
Adding the equality%
\begin{equation*}
\langle s_{1}-s_{2},\overline{x}\rangle =\frac{1}{2}(\Vert s_{1}\Vert
^{2}-\Vert s_{2}\Vert ^{2}),
\end{equation*}%
which follows from the fact that $\overline{x}\in H,$ to the preceding one,
we get%
\begin{equation*}
\langle t_{0}-s_{2},\overline{x}\rangle =\frac{1}{2}(\Vert t_{0}\Vert
^{2}-\Vert s_{2}\Vert ^{2}).
\end{equation*}%
Since $\interior V_{T}(S)\neq \emptyset $ and $\overline{x}\in \ri F$, there
is exactly one halfspace among those defined by the inequalities %
\eqref{eq:linrep} such that $\overline{x}$ belongs to its boundary
hyperplane (and hence to the interior of the remaining halfspaces). Hence
the linear equalities
\begin{equation}
\langle t_{0}-s_{1},x\rangle =\frac{1}{2}(\Vert t_{0}\Vert ^{2}-\Vert
s_{1}\Vert ^{2})  \label{eq:sp1}
\end{equation}%
and%
\begin{equation}
\langle t_{0}-s_{2},x\rangle =\frac{1}{2}(\Vert t_{0}\Vert ^{2}-\Vert
s_{2}\Vert ^{2})  \label{eq:sp2}
\end{equation}%
define the same hyperplane, which implies the existence of a real number $%
\lambda $ such that $t_{0}-s_{2}=\lambda \left( t_{0}-s_{1}\right) ,$ so
that the points $t_{0},s_{1},s_{2}$ are colinear. Substituting $x:=\frac{%
t_{0}+s_{1}}{2},$ which is a solution of \eqref{eq:sp1}, into \eqref{eq:sp2}
we have, after elementary algebraic manipulation,
\begin{equation*}
\langle t_{0}-s_{2},s_{1}-s_{2}\rangle =0,
\end{equation*}%
meaning that $s_{1}-s_{2}$ must be orthogonal to $t_{0}-s_{2}$. This,
together with the colinearity of $t_{0},s_{1}$ and $s_{2}$ and the fact that
these three points ar distinct, yields a contradiction.
\end{proof}

\section{Case Study}

\label{sec:specialcases}In this section we study the special case of higher
order cells on no more than four sites. With the exception of subsections
\ref{Halfspaces}-\ref{Wedges}, our study will be developed for sets in
$\R^{2}.$ For every set $F\subset\R^{2}$ defined by four
linear inequalities, we will determine whether or not there exist sets
$S\subset T\subset\R^{2},$ with $|S|=2$ and $|T|=4,$ such that
$V_{T}(S)=F.$ In the cases when the answer will be affirmative, we will
construct the (possibly non necessarity unique) sets $S$ and $T$ explicitly.

\subsection{Singletons}

A singleton (zero-dimensional) cell $\{c\}$ can be obtained by placing the
pairs of points $(s_{1},s_{2})$ and $(t_{1},t_{2})$ in the opposite corners of
a square centred at $c$. This was already discussed in
Example~\ref{eg:singleton}. Note that this is a minimal representation, since
we need at least three inequalities to obtain a bounded cell
(cf.~Theorem~\ref{thm:cardbound}), and hence $|T|\geq4$. In fact the square
can be replaced by a rectangle or a general \emph{cyclic} quadrilateral (we
recall that a quadrilateral is said to be cyclic if all of its vertices are on
a single circle, which is equivalent to the fact that the sum of
opposite angles equals $\pi$).

\begin{proposition}
\label{prop:cyclic} Let $S\subset T\subset\R^{2},$ with $|S|=2$ and
$|T|=4.$ The following statements are equivalent:

\begin{itemize}
\item[a)] $V_{T}(S)$ is nonempty and at most one-dimensional.

\item[b)] The points of $T$ are the vertices of a cyclic quadrilateral, with
the two sites of $S$ located opposite to each other (across a diagonal).

\item[c)] $V_{T}(S)$ is a singleton.
\end{itemize}
\end{proposition}

\begin{proof}
Throughout the proof, we use the explicit notation $S:=\{s_{1},s_{2}\}$ and $%
T:=\{s_{1},s_{2},t_{1},t_{2}\}$.
a) $\Rightarrow $ b). Without loss of generality assume that $0\in V_{T}(S)$
while $\interior V_{T}(S)=\emptyset $. By Theorem~\ref{thm:interior} we have
\begin{equation}
0=\sum_{i,j\in \{1,2\}}\lambda _{ij}(t_{j}-s_{i}),\qquad 0=\sum_{i,j\in
\{1,2\}}\lambda _{ij}(\Vert t_{j}\Vert ^{2}-\Vert s_{i}\Vert ^{2}),
\label{eq:sysspecial}
\end{equation}%
where $\lambda _{ij}$ are convex combination coefficients. Since $0\in
V_{T}(S)$, we have from \eqref{eq:linrep} that
\begin{equation*}
\Vert t_{j}\Vert \geq \Vert s_{i}\Vert \quad \forall i,j\in \{1,2\}.
\end{equation*}%
Without loss of generality assume that
\begin{equation*}
\Vert s_{2}\Vert \leq \Vert s_{1}\Vert \leq \Vert t_{1}\Vert \leq \Vert
t_{2}\Vert .
\end{equation*}%
If $\Vert s_{1}\Vert <\Vert t_{1}\Vert $, then $\Vert s_{i}\Vert <\Vert
t_{j}\Vert $ for all $i,j\in \left\{ 1,2\right\} ,$ which implies that the
second equality in \eqref{eq:sysspecial} is impossible. Hence, $\Vert
s_{1}\Vert =\Vert t_{1}\Vert $. If $\Vert s_{1}\Vert >\Vert s_{2}\Vert $,
then $\Vert t_{i}\Vert >\Vert s_{2}\Vert $ for all $i\in \left\{ 1,2\right\}
;$ therefore, from (\ref{eq:sysspecial}), we get $\lambda _{21}=\lambda
_{22}=0$ and
\begin{equation*}
s_{1}=\lambda _{11}t_{1}+\lambda _{12}t_{2},
\end{equation*}%
so $s_{1}\in \lbrack t_{1},t_{2}]$ and $\Vert s_{1}\Vert ^{2}=\lambda
_{11}\Vert t_{1}\Vert ^{2}+\lambda _{12}\Vert t_{2}\Vert ^{2}$, which holds
only when $t_{1}=t_{2}=s_{1}$ by the strict convexity of the squared norm.
This is impossible. Likewise, when $\Vert t_{1}\Vert <\Vert t_{2}\Vert $ we
have $\lambda _{12}=\lambda _{22}=0$, then $t_{1}\in \lbrack s_{1},s_{2}]$,
which by Corollary~\ref{prop:emptycell} yields $V_{T}(S)=\emptyset $, again
a contradiction. We have proved that $\Vert s_{2}\Vert =\Vert s_{1}\Vert
=\Vert t_{1}\Vert =\Vert t_{2}\Vert $, and hence our sites are the vertices
of a cyclic quadrilateral. It is now easy to observe that $s_{1},s_{2}$ are
located opposite to each other because, by the first equality in %
\eqref{eq:sysspecial}, we have $[s_{1},s_{2}]\cap \lbrack t_{1},t_{2}]\neq
\emptyset $. b) $\Rightarrow $ c) Since the points of $T$ lie on some
circle, without loss of generality we may assume that the centre of
this circle is the origin, and then $\Vert t_{1}\Vert =\Vert
t_{2}\Vert =\Vert s_{1}\Vert =\Vert s_{2}\Vert $. In this case the
right-hand side of system \eqref{eq:linrep} is zero, and we have, for every
point $x\in V_{T}(S),$
\begin{equation}
\langle t_{j}-s_{i},x\rangle \leq 0\quad \forall i,j\in \{1,2\}.
\label{homogeneous}
\end{equation}%
It is evident that $x=0$ is a solution of this system; hence, $V_{T}(S)\neq
\emptyset $. On the other hand, from (\ref{homogeneous}) it follows that $%
V_{T}(S)$ is a cone. From these facts, using Proposition \ref%
{prop:intersectfour} we immediately deduce that $V_{T}(S)=\left\{ 0\right\}
. $ The implication c) $\Rightarrow $ a) is obvious.
\end{proof}

\bigskip

Note that for the case $|S|=2$ and $|T|=3$ it is impossible to have a nonempty
bounded cell due to Remark \ref{Remark 8}. This means that we do not need to
consider this configuration when discussing the subsequent cases of bounded polygons.

Furthermore, in the case $|S|=3$ and $|T|=4$ it is impossible to have a
bounded cell, as was shown in Proposition~\ref{prop:four-three}.

Since we have determined that we can not have a nonempty bounded cell for
$|T|=|S|+1$, the only possibility to have a singleton cell is for $|S|=2$ and
$|T|=4$. Furthermore, we can focus on the latter case when studying other
bounded cells.

\subsection{One-Dimensional Cells}

It follows from the preceding discussion that it is impossible to obtain line
segments as multipoint Voronoi cells in our setting.

\begin{corollary}
\label{cor:segmimp} Let $S\subset T\subset\R^{2},$ with $|S|=2$ and
$|T|=4.$ Then $V_{T}(S)$ is not one-dimensional.
\end{corollary}

\begin{proof}
Follows directly from Proposition~\ref{prop:cyclic}.
\end{proof}

It follows from Corollary~\ref{cor:segmimp} that it is impossible to have a
one-dimensional cell for $|T|=4$, $|S|=2$, so both rays and lines are
impossible in this configuration.

Now consider the case $|T|=|S|+1$. If $V_{T}\left(  S\right)  \neq\emptyset,$
by Corollary~\ref{prop:emptycell} we must have for $\{t\}=T\setminus S$ that
$t\notin\conv S$. By the separation theorem this yields the existence of some
$d$, $\Vert d\Vert=1$ such that
\[
\langle t-s,d\rangle<0\quad\forall s\in S,
\]
which yields the existence of a sufficiently small ball $B_{\varepsilon}(d)$
centred at $d$ such that
\[
\langle t-s,y\rangle<0\quad\forall s\in S,\;\forall y\in B_{\varepsilon}(d).
\]
Then for any $x_{0}\in V_{T}(S)$ and any $y\in B_{\varepsilon}(d)$ we have
\[
\langle t-s,x_{0}+y\rangle=\langle t-s,x_{0}\rangle+\langle t-s,y\rangle
<\langle t-s,x_{0}\rangle.
\]
It is hence clear from the representation in
Proposition~\ref{prop:inequalities} that $x_{0}+B_{\varepsilon}(y)\subset
V_{T}(S)$, which gives that $V_{T}(S)$ is one-dimensional.

Thus we have proved the following statement.

\begin{proposition}
Let $S\subset T\subset\R^{n},$ with $|T|= |S|+1$. Then $V_{T}(S)$ is
not one-dimensional.
\end{proposition}

\subsection{Triangles}

A somewhat surprising result is that a second order Voronoi cell cannot be a
triangle. As discussed previously, we only need to prove this for the case
$|S|=2$ and $|T|=4$.

\begin{proposition}
\label{triangles}Let $S\subset T\subset\R^{2},$ with $|S|=2$ and
$|T|=4.$ Then $V_{T}(S)$ is not a triangle.
\end{proposition}

\begin{proof}
Suppose that $V_{T}\left( S\right) $ is a triangle. Denote $%
T=\{s_{1},s_{2},t_{1},t_{2}\}$, $S=\{s_{1},s_{2}\}$. The cell $V_{T}(S)$ is
the solution set of the linear system of inequalities
\begin{equation}
\left\langle c_{ij},x\right\rangle \leq \alpha _{ij}\text{\quad }\left(
i,j=1,2\right) ,  \label{system}
\end{equation}%
with $c_{ij}:=t_{i}-s_{j}$ and $\alpha _{ij}:=\frac{1}{2}\left( \left\Vert
t_{i}\right\Vert ^{2}-\left\Vert s_{j}\right\Vert ^{2}\right) .$ Notice that%
\begin{equation}
\left(
\begin{array}{c}
c_{22} \\
\alpha _{22}%
\end{array}%
\right) =-\left(
\begin{array}{c}
c_{11} \\
\alpha _{11}%
\end{array}%
\right) +\left(
\begin{array}{c}
c_{12} \\
\alpha _{12}%
\end{array}%
\right) +\left(
\begin{array}{c}
c_{21} \\
\alpha _{21}%
\end{array}%
\right) .  \label{relation}
\end{equation}%
Without loss of generality, we will assume that $0\in \interior V_{T}\left(
S\right) $, that is, $\alpha _{ij}>0$ $\left( i,j=1,2\right) .$ Since (\ref%
{system}) defines a triangle, one of its inequalities, say the one
corresponding to $i=2=j$ is redundant, so that the triangle is actually the
solution set of the system consisting of the other three. We claim that the
vectors $\left(
\begin{array}{c}
c_{11} \\
\alpha _{11}%
\end{array}%
\right) ,\left(
\begin{array}{c}
c_{12} \\
\alpha _{12}%
\end{array}%
\right) $ and $\left(
\begin{array}{c}
c_{21} \\
\alpha _{21}%
\end{array}%
\right) $ are linearly independent. Indeed, assume the existence of $\left(
\beta _{11},\beta _{12},\beta _{21}\right) \in \R^{3}\setminus
\left\{ \left( 0,0,0\right) \right\} $ such that%
\begin{equation}
\beta _{11}\left(
\begin{array}{c}
c_{11} \\
\alpha _{11}%
\end{array}%
\right) +\beta _{12}\left(
\begin{array}{c}
c_{12} \\
\alpha _{12}%
\end{array}%
\right) +\beta _{21}\left(
\begin{array}{c}
c_{21} \\
\alpha _{21}%
\end{array}%
\right) =\left(
\begin{array}{c}
0_{2} \\
0%
\end{array}%
\right) ,  \label{linear independence}
\end{equation}%
and let $\overline{x}$ be the vertex of $V_{T}(S)$ defined by%
\begin{equation}
\left\langle c_{12},\overline{x}\right\rangle =\alpha _{12}\text{ and }%
\left\langle c_{21},\overline{x}\right\rangle =\alpha _{21}.  \label{vertex}
\end{equation}%
Then, from (\ref{linear independence}) and (\ref{vertex}) we easily deduce
that $\beta _{11}\left( \left\langle c_{11},\overline{x}\right\rangle
-\alpha _{11}\right) =0.$ Since $\left\langle c_{11},\overline{x}%
\right\rangle <\alpha _{11},$ as the three sides of $V_{T}(S)$ don't have a
common point, it follows that $\beta _{11}=0.$ By using the same argument
with the other two vertices of $V_{T}(S),$ we deduce that $\beta
_{12}=0=\beta _{21},$ thus proving our claim that $\left(
\begin{array}{c}
c_{11} \\
\alpha _{11}%
\end{array}%
\right) ,\left(
\begin{array}{c}
c_{12} \\
\alpha _{12}%
\end{array}%
\right) $ and $\left(
\begin{array}{c}
c_{21} \\
\alpha _{21}%
\end{array}%
\right) $ are linearly independent. Since the inequality corresponding to $%
i=2=j$ is redundant, by Farkas' lemma there exist $\lambda _{ij}\geq 0$ $%
\left( i,j=1,2\right) $ such that
\begin{equation}
c_{22}=\lambda _{11}c_{11}+\lambda _{12}c_{12}+\lambda _{21}c_{21}\quad
\text{and\quad }\alpha _{22}\geq \lambda _{11}\alpha _{11}+\lambda
_{12}\alpha _{12}+\lambda _{21}\alpha _{21}.  \label{Farkas}
\end{equation}%
Considering again the vertex $\overline{x},$ from (\ref{relation}) we obtain
that%
\begin{align*}
\alpha _{22}& =-\alpha _{11}+\alpha _{12}+\alpha _{21}\leq -\left\langle
c_{11},\overline{x}\right\rangle +\left\langle c_{12},\overline{x}%
\right\rangle +\left\langle c_{21},\overline{x}\right\rangle =\left\langle
-c_{11}+c_{12}+c_{21},\overline{x}\right\rangle \\
& =\left\langle c_{22},\overline{x}\right\rangle \leq \alpha _{22};
\end{align*}%
hence, by (\ref{Farkas}), we deduce that%
\begin{align*}
\left\langle c_{22},\overline{x}\right\rangle & =\alpha _{22}\geq \lambda
_{11}\alpha _{11}+\lambda _{12}\alpha _{12}+\lambda _{21}\alpha _{21} \\
& \geq \lambda _{11}\left\langle c_{11},\overline{x}\right\rangle +\lambda
_{12}\left\langle c_{12},\overline{x}\right\rangle +\lambda
_{21}\left\langle c_{21},\overline{x}\right\rangle =\left\langle \lambda
_{11}c_{11}+\lambda _{12}c_{12}+\lambda _{21}c_{21},\overline{x}\right\rangle
\\
& =\left\langle c_{22},\overline{x}\right\rangle .
\end{align*}%
Therefore $\lambda _{11}\alpha _{11}+\lambda _{12}\alpha _{12}+\lambda
_{21}\alpha _{21}=\alpha _{22},$ that is,%
\begin{equation*}
\left(
\begin{array}{c}
c_{22} \\
\alpha _{22}%
\end{array}%
\right) =\lambda _{11}\left(
\begin{array}{c}
c_{11} \\
\alpha _{11}%
\end{array}%
\right) +\lambda _{12}\left(
\begin{array}{c}
c_{12} \\
\alpha _{12}%
\end{array}%
\right) +\lambda _{12}\left(
\begin{array}{c}
c_{21} \\
\alpha _{21}%
\end{array}%
\right) .
\end{equation*}%
Comparing this equality with (\ref{relation}) and taking into account that
the vectors $\left(
\begin{array}{c}
c_{11} \\
\alpha _{11}%
\end{array}%
\right) ,\left(
\begin{array}{c}
c_{12} \\
\alpha _{12}%
\end{array}%
\right) $ and $\left(
\begin{array}{c}
c_{21} \\
\alpha _{21}%
\end{array}%
\right) $ are linearly independent, we deduce that $\lambda _{11}=-1,$ a
contradiction.
\end{proof}

\subsection{Bounded Quadrilaterals}

\label{sec:quadrilaterals}

We next show that any non-cyclic bounded convex quadrilateral is a second
order Voronoi cell with $|T|=4$. We also prove that it is impossible for a
cyclic quadrilateral to be a Voronoi cell of 2 points (when $|T|=4$).

\begin{proposition}
\label{prop:quadrilateral} Let $F\subset\R^{2}$ be a non-cyclic
bounded convex quadrilateral. Then there exist {$S\subset T\subset
\R^{2}$}, with $|S|=2$ and $|T|=4$, such that $V_{T}(S)=F$.
\end{proposition}

\begin{proof}
First, looking at Fig.~\ref{fig:quadrilateral-second}, where $F$ is depicted
as the quadrilateral $ACBD,$%
\begin{figure}[th]
\centering \includegraphics[width=0.8\textwidth]{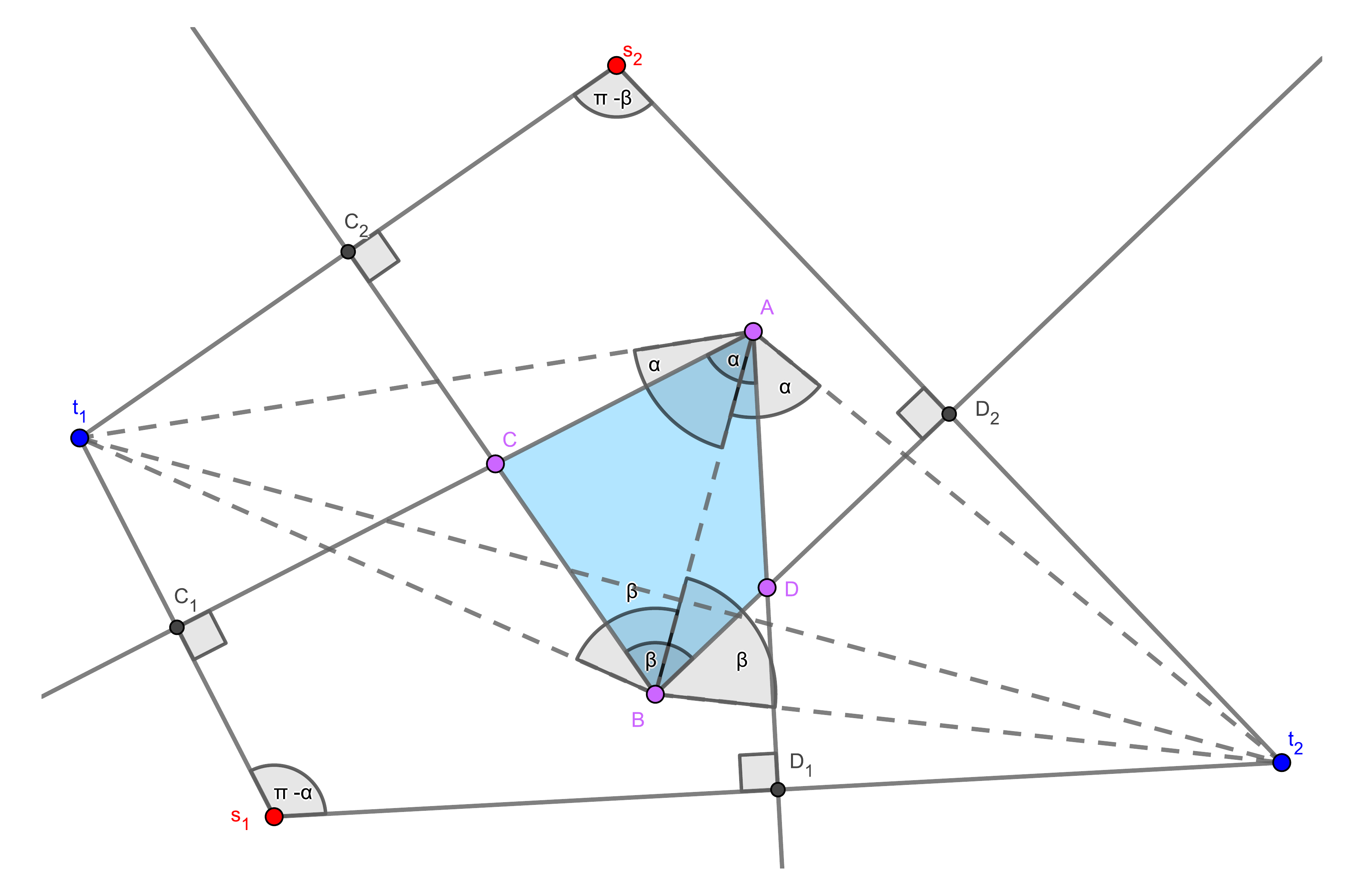}
\caption{Bounded quadrilateral.}
\label{fig:quadrilateral-second}
\end{figure}
we would like to mention that $\sphericalangle CAD=\alpha >0$ and $%
\sphericalangle CBD=\beta >0$ and, w.l.o.g., $\pi >\alpha +\beta .$ We draw
the lines through $A$ ($B$) that make an angle equal to $\alpha $ (resp., $%
\beta $) with the segment $AB.$ In this way, we get the points $t_{1}$ and $%
t_{2}$ as the two other vertices of the quadrilateral having those lines as
their sides. We take $s_{1}$ and $s_{2}$ as the symmetric points of $t_{1}$
with respect to the lines $CA$ and $CB,$ respectively, and we set $%
S:=\left\{ s_{1},s_{2}\right\} $ and $T:=\left\{
s_{1},s_{2},t_{1},t_{2}\right\} $. The only thing that we have to prove is
that $t_{2}s_{1}\bot AD$ and $t_{2}s_{2}\bot BD.$ The point $B$ is the
center of the circle passing through the points $t_{1},$ $t_{2}\ $and$%
\ s_{2}.$ Therefore, $\sphericalangle t_{1}s_{2}t_{2}=\pi -\beta ,$ which,
looking at the quadrilateral $C_{2}s_{2}D_{2}B,$ implies that $%
t_{2}s_{2}\bot BD.$ The point $A$ is the center of the circle passing
through the points $t_{1},$ $t_{2}$ and$\ s_{1}.$ Therefore, $\sphericalangle
t_{1}s_{1}t_{2}=\pi -\alpha ,$ which, looking at the quadrilateral $%
C_{1}s_{1}D_{1}A,$ implies $t_{2}s_{1}\bot AD.$
\end{proof}

\bigskip

Observe that the algorithm does not work if $\pi\leq\alpha+\beta$. However it
is not difficult to observe that for a non-cyclic convex quadrilateral it is
always possible to choose the corners to ensure $\alpha+\beta<\pi$. We have
the following negative result.

\begin{proposition}
\label{prop:cyclicneg} Let $S\subset T\subset\R^{2},$ with $|S|=2$ and
$|T|=4.$ Then $V_{T}(S)$ is not a cyclic quadrilateral.
\end{proposition}

\begin{proof}
Assume that the Voronoi cell $V_{T}(S)$ of some set $S:=\{s_{1},s_{2}\}$,
with $T:=\{s_{1},s_{2},t_{1},t_{2}\}$, is a cyclic quadrilateral. Then each
side of this quadrilateral is defined by the bisector between $s_{i}$ and $%
t_{j}$ for $i,j\in \{1,2\}$. First, we will show that any two sides defined
by the bisectors of disjoint pairs, say, $\{s_{1},t_{1}\}$ and $%
\{s_{2},t_{2}\}$, cannot be adjoint. Assume the contrary: then without loss
of generality the intersection $u$ of the two bisectors is a vertex of $%
V_{T}(S)$. We have, by using the representation \eqref{eq:linrep},
\begin{equation}
\langle t_{1}-s_{2},u\rangle \leq \frac{1}{2}(\Vert t_{1}\Vert ^{2}-\Vert
s_{2}\Vert ^{2}),\qquad \langle t_{2}-s_{1},u\rangle \leq \frac{1}{2}(\Vert
t_{2}\Vert ^{2}-\Vert s_{1}\Vert ^{2}),  \label{eq:constrfu}
\end{equation}%
and since $u$ is the intersection of the two bisectors,
\begin{equation}
\langle t_{1}-s_{1},u\rangle =\frac{1}{2}(\Vert t_{1}\Vert ^{2}-\Vert
s_{1}\Vert ^{2}),\qquad \langle t_{2}-s_{2},u\rangle =\frac{1}{2}(\Vert
t_{2}\Vert ^{2}-\Vert s_{2}\Vert ^{2}).  \label{eq:defu}
\end{equation}%
Adding the two equalities in \eqref{eq:defu} and rearranging, we obtain
\begin{equation*}
\langle t_{1}-s_{2},u\rangle +\langle t_{2}-s_{1},u\rangle =\frac{1}{2}%
(\Vert t_{1}\Vert ^{2}-\Vert s_{2}\Vert ^{2})+\frac{1}{2}(\Vert t_{2}\Vert
^{2}-\Vert s_{1}\Vert ^{2}).
\end{equation*}%
Together with \eqref{eq:constrfu} this yields equalities in %
\eqref{eq:constrfu}, and hence the four lines that define the sides of the
quadrilateral must intersect at $u$. This is impossible, hence, the
assumption is wrong.
\begin{figure}[th]
\centering \includegraphics[width=0.7\textwidth]{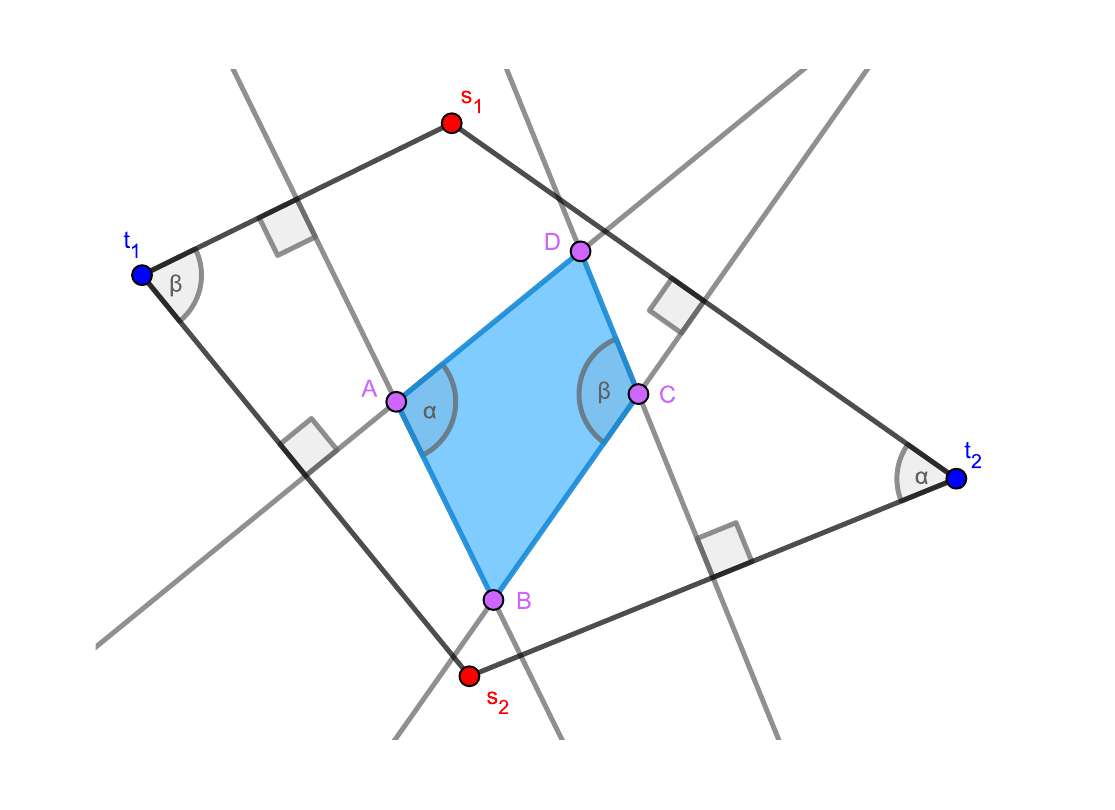}
\caption{An illustration to the proof of Proposition~\protect\ref%
{prop:cyclicneg}.}
\label{fig:impossible-cyclic}
\end{figure}
Now, let us consider the quadrilateral with vertices $s_{1},s_{2},t_{1}$ and
$t_{2}.$ Looking at Fig.~\ref{fig:impossible-cyclic}, it is easy to see that
the angles at $t_{1}$ and $C$ are equal, and so they are the angles at $%
t_{2} $ and $A.$ This means that this quadrilateral is cyclic too, that is, $%
s_{1},s_{2},t_{1}$ and $t_{2}$ lie on a circle. It follows from
Proposition~\ref{prop:cyclic} that the Voronoi cell $V_{T}(S)$ is a
singleton, which contradicts our assumption.
\end{proof}

\subsection{Halfspaces\label{Halfspaces}}

A halfspace cell can be obtained by putting the two points of $S$ on a line
perpendicular to the boundary line of the halfspace making sure that $S$ is in
the interior of the halfspace. An additional point $t$ is placed on the same
line on the opposite side of the hyperplane at the same distance from the
hyperplane as the distance to the hyperplane from the farthest point in $S$
(see Fig.~\ref{fig:half-space}). \begin{figure}[th]
\centering \includegraphics[width=0.7\textwidth]{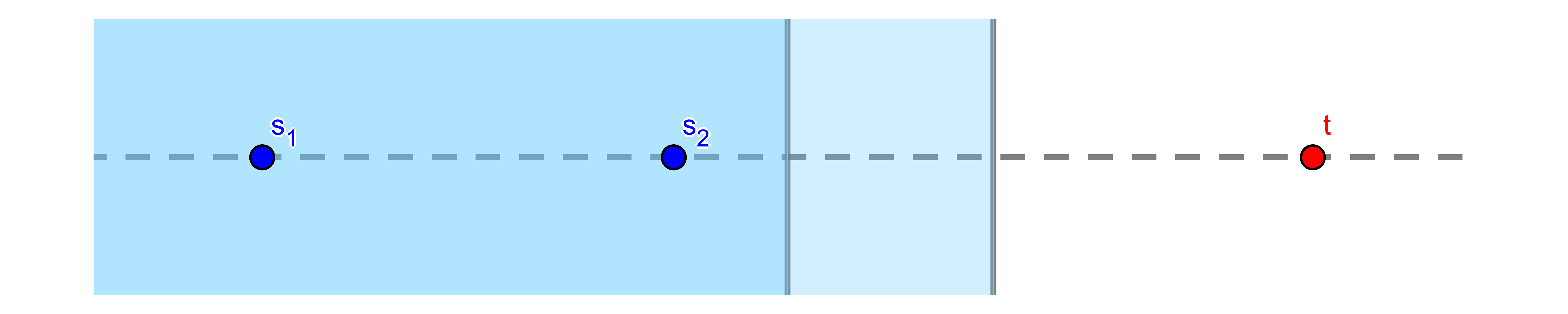}\caption{The
intersection of the two halfspaces is $V_{T}({s_{1},s_{2}})$.}%
\label{fig:half-space}%
\end{figure}We hence conclude that a halfspace can be constructed using
$|S|=2$ and $|T|=3$. Observe that it is also possible to do the same
construction for $|S|\in\{2,3\}$ and $|T|=4$. We prove this explicitly in the
next statement.

\begin{proposition}
\label{prop:half-space}

Let $F\subset\R^{n}$ be a halfspace. Then for any two integers
$\tau>\sigma\geq1$ there exist {$S\subset T\subset\R^{n}$}, with
$|S|=\sigma$ and $|T|=\tau$, such that $V_{T}(S) = F$.
\end{proposition}

\begin{proof}
Note that any halfspace $F$ can be represented as $F=\{x\in \R%
^{n}\,|\,\langle d,x\rangle \leq \gamma \}$ for some $d$, $\Vert d\Vert =1,$
and $\gamma \in \R$.
Choose any $x_{0}$ such that $\langle d,x_{0}\rangle =\gamma $, and let $%
S:=\{s_{0},\dots ,s_{m}\}$ and $T:=\{t_{0},\dots ,t_{p}\}\cup S$, where $%
m=\sigma -1$, $p=\tau -\sigma -1$,
\begin{equation*}
s_{0}:=x_{0}-d,\text{ }t_{0}:=x_{0}+d,
\end{equation*}%
\begin{equation}
s_{i}=x_{0}-\alpha _{i}d,\;0<\alpha _{i}<1\;\forall i\in \{1,\dots
,m\},\; t_{j}:=x_{0}+\beta _{j}d,\,\beta _{j}>1\; \forall j\in
\{1,\dots ,p\},  \label{facts}
\end{equation}%
and the constants $\alpha _{i}$ and $\beta _{j}$ are all different (to
ensure that the sites do not coincide). We will next show that
\begin{equation*}
V_{T}(S)=\{x\in \R^{2}\,|\,\langle d,x\rangle \leq \gamma \}.
\end{equation*}%
We have, from the representation in Proposition~\ref{prop:inequalities},
\begin{equation*}
V_{T}(S)=\bigcap_{\substack{ i\in \{0,1,\dots ,m\}  \\ j\in \{0,1,\dots ,p\}
}}\left\{ x\,:\,\langle t_{j}-s_{i},x\rangle \leq \frac{1}{2}\left( \Vert
t_{j}\Vert ^{2}-\Vert s_{i}\Vert ^{2}\right) \right\} .
\end{equation*}%
By (\ref{facts}), the inequalities in the latter expression can be rewritten
as
\begin{equation*}
(\beta _{j}+\alpha _{i})\langle d,x\rangle \leq \frac{1}{2}\left( (\beta
_{j}-\alpha _{i})(\beta _{j}+\alpha _{i})+(\beta _{j}+\alpha _{i})\langle
d,x_{0}\rangle \right) ,\; \begin{array}{l}\forall i\in \{0,\dots ,m\},\\\forall
j\in \{0,\dots ,p\},\end{array}
\end{equation*}%
where $\alpha _{0}=\beta _{0}=1$. Dividing by the factor $\beta _{j}+\alpha
_{i}>0$, we have
\begin{equation*}
\langle d,x\rangle -\langle d,x_{0}\rangle \leq \frac{\beta _{j}-\alpha _{i}%
}{2};
\end{equation*}%
hence,%
\begin{equation*}
V_{T}(S)=\left\{ x\,:\,\langle d,x\rangle -\gamma \leq \frac{1}{2}\cdot
\min_{i,j}(\beta _{j}-\alpha _{i})\right\} =\{x\,:\,\langle d,x\rangle \leq
\gamma \}.
\end{equation*}%
The latter set is precisely the halfspace that we were aiming for.
\end{proof}

\subsection{Intersections of Parallel Halfspaces}

We consider a set $F\subset\R^{n}$ represented by the inequalities
$\alpha\leq\left\langle d,x\right\rangle \leq\beta,$ with $\left\Vert
d\right\Vert =1$ and $\alpha<\beta.$ One can easily check that $F=V_{T}\left(
S\right)  $ for $T:=\left\{  s_{1},s_{2},t_{1},t_{2}\right\}  $ and
$S:=\left\{  s_{1},s_{2}\right\}  $ with, for instance, $s_{1}:=\frac
{3\alpha+\beta}{4}d,$ $s_{2}:=\frac{\alpha+\beta}{2}d,$ $t_{1}:=\frac
{3\alpha-\beta}{2}d$ and $t_{2}:=\frac{7\beta-3\alpha}{4}d.$ We have shown the
following result.

\begin{proposition}
\label{prop:strip} Let $F\subset\R^{n}$ be an intersection of two
parallel halfspaces with opposite normals, such that $F$ has a nonempty
interior. Then there exist {$S\subset T\subset\R^{n}$}, with $|S|=2$
and $|T|=4$, such that $V_{T}(S)=F$.
\end{proposition}

Note that it is impossible to produce a strip with $|T|-|S|=1$: indeed, it is
clear from the fact that all constraints are defined by parallel lines that
all vectors $s-t$, $s\in S$, $t\in T$ should be colinear, lying on some line
orthogonal to the inequalities. Now, if for the unique $t\in T\setminus S$ we
have $t\in\conv S$, then the cell is empty by Corollary~\ref{prop:emptycell}.
However if $t\notin\conv S$, then the cell has to be a halfspace. Hence the
only possibility is $|T|=4$, $|S|=2$.

\subsection{Wedges\label{Wedges}}

Now, let us have a wedge $F:=\left\{  x\in\R^{n}:\left\langle
c_{i},x\right\rangle \leq\alpha_{i},\text{ }i=1,2\right\}  ,$ with $c_{1}$ and
$c_{2}$ being linearly independent$.$ Let us take an arbitrary point from
$t\in\interior F^{-}:=\left\{  x\in\R^{n}:\left\langle c_{i}%
,x\right\rangle \geq\alpha_{i},\text{ }i=1,2\right\}  $ and construct the two
symmetric points $s_{1}$ and $s_{2}$\ with respect to both hyperplanes
defining $F.$ If $T:=\left\{  s_{1},s_{2},t\right\}  $ and $S:=\left\{
s_{1},s_{2}\right\}  ,$ then $V_{T}\left(  S\right)  =F_{T}(t)=F.$%
\begin{figure}[th]
\centering \includegraphics[width=0.8\textwidth]{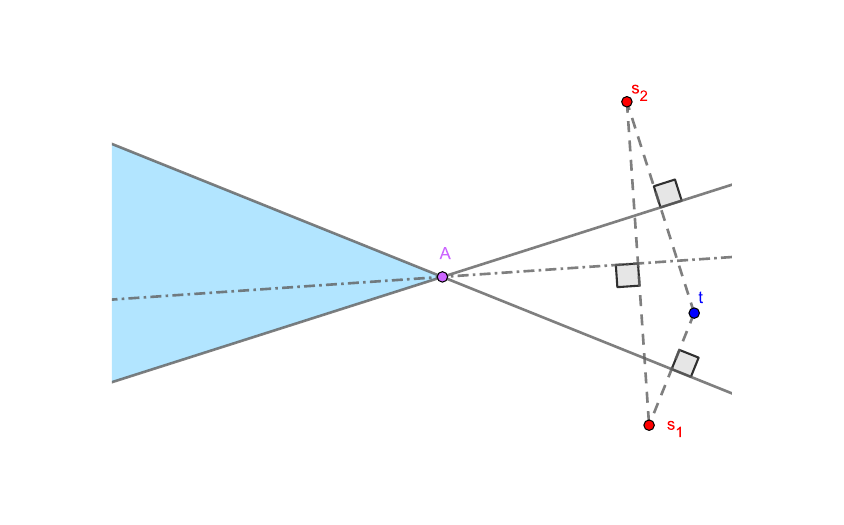}\caption{Construction
of the set of sites for a wedge}%
\label{fig:wedge-construction}%
\end{figure}

It is possible to add any number of more sites $s$ to $S$ in such a way that
the halfspaces $\Vert s-x\Vert\leq\Vert t-x\Vert$ include the original wedge.

\begin{proposition}
\label{prop:wedge-four} Let
\[
F:=\left\{  x\in\R^{n}\,|\,\langle v_{i},x\rangle\leq b_{i},\quad i=1,2\right\}  ,
\]
with $v_{1,}v_{2}\in\R^{n}$ linearly independent unit vectors and
$b_{1},b_{2}\in\R.$ Then there exist $S\subset T\subset\R%
^{n},$ with $|S|=2$ and $|T|=4,$ such that $V_{T}(S)=F$.
\end{proposition}

\begin{proof}
Without loss of generality, we assume that $b_{i}=0$ $\left( i=1,2\right) .$
Set $\beta :=\left\langle v_{1},v_{2}\right\rangle ,$ $i^{\prime }:=3-i$ for
$i=1,2,$ and define
\begin{equation*}
s_{i}:=-v_{i}-2v_{i^{\prime }}\text{\qquad }\left( i=1,2\right) .
\end{equation*}%
%
%
%
%
%
%
%
%
Let $t_{i}$ be the point symmetric to $s_{i}$ with respect to the hyperplane
defined by $\langle x,v_{i^{\prime }}\rangle =0,$ that is,
\begin{equation*}
t_{i}:=-v_{i}+2\left( 1+\beta \right) v_{i^{\prime }}\text{\qquad }\left(
i=1,2\right) ,
\end{equation*}%
and define $S:=\left\{ s_{1},s_{2}\right\} $ and $T:=\left\{
s_{1},s_{2},t_{1,}t_{2}\right\} .$ We have $|T|=4,$ since $\beta >-1.$ Given
that
\begin{equation*}
\left\Vert t_{i}\right\Vert ^{2}=5+4\beta =\Vert s_{i}\Vert ^{2}\text{\qquad
}\left( i=1,2\right) ,
\end{equation*}%
one has
\begin{equation}
V_{T}(S)=\left\{ x\in \R^{n}\,|\,\langle t_{i}-s_{i},x\rangle \leq
0,\;\langle t_{i}-s_{i^{\prime }},x\rangle \leq 0,\text{\qquad }%
i=1,2\right\} .  \label{cell}
\end{equation}%
Hence, to prove that $V_{T}(S)=F,$ it will suffice to show that, for $i=1,2,$
the inequality $\langle x,t_{i}-s_{i}\rangle \leq 0$ is equivalent to $%
\langle x,v_{i^{\prime }}\rangle \leq 0$ and that the remaining two
inequalities in (\ref{cell}) are redundant. The first assertion is obvious,
since $t_{i}-s_{i}=2\left( 2+\beta \right) v_{i^{\prime }}$ and $\beta >-1.$
Let us now prove that the inequalities $\langle t_{i}-s_{i^{\prime
}},x\rangle \leq 0$ are redundant, that is, that they are consequences of
the system $\langle v_{i},x\rangle \leq 0$ $\left( i=1,2\right) .$ But this
is also immediate, since $t_{i}-s_{i^{\prime }}=v_{i}+\left( 3+2\beta
\right) v_{i^{\prime }}$ and $\beta >-1.$
\end{proof}

\subsection{Unbounded Polygons with Three Sides}

Our next construction is for an unbounded polygon like the shadowed one in
Fig.~\ref{fig:unbounded-pic2-construction}, \begin{figure}[th]
\centering \includegraphics[width=0.7\textwidth]{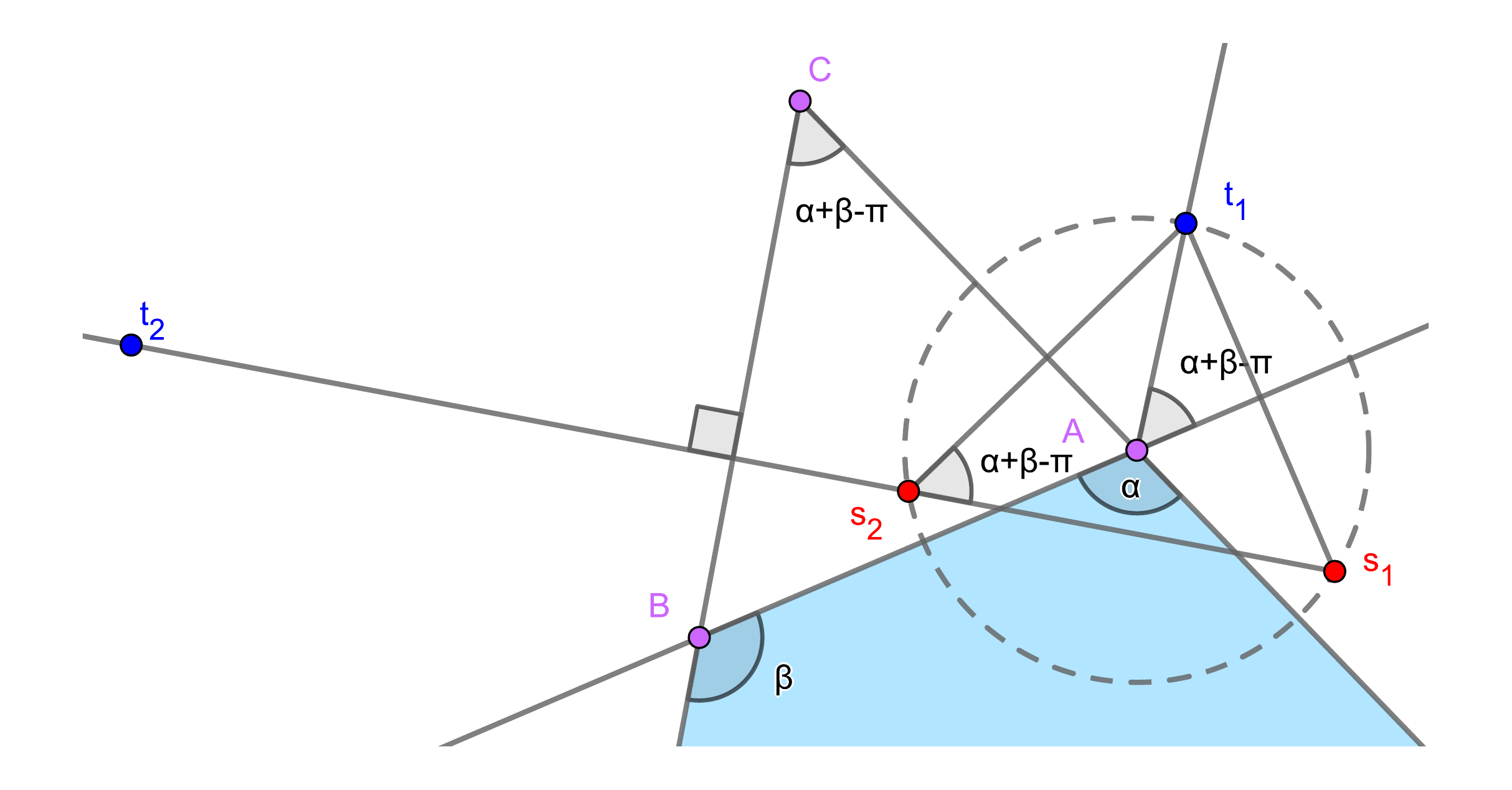}\caption{Unbounded
polygon with three sides}%
\label{fig:unbounded-pic2-construction}%
\end{figure}with three sides, which is unbouned and has non-parallel sides.

\begin{proposition}
\label{prop:alg-one} Let $F$ be an unbounded convex polygon with non-parallel
sides and two vertices. Then there exist {$S\subset T\subset\R^{2}$},
with $|S|=2$ and $|T|=4$, such that $V_{T}(S)=F$.
\end{proposition}

\begin{proof}
In the notation of Fig.~\ref{fig:unbounded-pic2-construction}, where $F$ is
depicted as the shaded region, we would like to mention that $0<\alpha <\pi $%
, $0<\beta <\pi $, $\alpha +\beta >\pi $ and $\beta \geq \alpha $. If we
consider the triangle $ABC,$ the angle at $C$ is $\alpha +\beta -\pi .$ Now,
we draw a line through the point $A,$ which makes an angle $\alpha +\beta
-\pi $ with the line $BA.$ On this line we consider an arbitrary point $%
t_{1} $, sufficiently close to $A.$ Let us take the symmetric points $s_{1}$
and $s_{2}$ of $t_{1}$ with respect to $AB$ and $AC,$ respectively. We shall
prove that the line through the points $s_{1}$, $s_{2}$ is perpendicular to $%
BC.$ For this purpose, we would like to mention that the point $A$ is the
center of the circle passing through the points $s_{1}$, $s_{2}$ and $%
t_{1}.$ This means that $\sphericalangle s_{1}s_{2}t_{1}=\alpha +\beta -\pi
. $ If we consider the quadrilateral formed by the lines $BC,$ $AC,$ $%
s_{1}s_{2}$ and $t_{1}s_{2},$ we get the desired fact. At the end we
construct the site $t_{2}$ as the symmetric point of $s_{1}$ with respect to
the line $BC\ $and set $S:=\left\{ s_{1},s_{2}\right\} $ and $T:=\left\{
s_{1},s_{2},t_{1},t_{2}\right\} $.
\end{proof}

\begin{proposition}
\label{prop:impunbthree} Let $S\subset T\subset\R^{2},$ with $|S|=2$
and $|T|=4.$ Then $V_{T}(S)$ is not an unbounded polygon with parallel sides
and just two vertices.
\end{proposition}

\begin{proof}
Similarly to the proof of Proposition \ref{triangles}, suppose that $%
V_{T}\left( S\right) $ is an unbounded convex polygon with parallel sides
and just two vertices, denote $S=\{s_{1},s_{2}\}$ and $T=%
\{s_{1},s_{2},t_{1},t_{2}\},$ and suppose further, w.l.o.g., that the
redundant inequality in the linear system
\begin{equation}
\left\langle t_{i}-s_{j},x\right\rangle \leq \frac{1}{2}\left( \left\Vert
t_{i}\right\Vert ^{2}-\left\Vert s_{j}\right\Vert ^{2}\right) \text{\quad }%
\left( i,j=1,2\right)  \label{system 2}
\end{equation}%
(which has $V_{T}(S)$ as its solution set) is the one corresponding to $%
i=2=j.$ Then, by Farkas' Lemma, there exist $\gamma _{11},\gamma
_{12},\gamma _{21}\geq 0$ such that%
\begin{equation}
t_{2}-s_{2}=\gamma _{11}\left( t_{1}-s_{1}\right) +\gamma _{12}\left(
t_{1}-s_{2}\right) +\gamma _{21}\left( t_{2}-s_{1}\right)  \label{cone}
\end{equation}%
and%
\begin{equation}
\alpha _{22}\geq \gamma _{11}\alpha _{11}+\gamma _{12}\alpha _{12}+\gamma
_{21}\alpha _{21},  \label{ineq}
\end{equation}%
where $\alpha _{ij}:=\frac{1}{2}\left( \left\Vert t_{i}\right\Vert
^{2}-\left\Vert s_{j}\right\Vert ^{2}\right) .$ From (\ref{cone}) we obtain
that $t_{2}-s_{2}$ belongs to%
\begin{equation*}
K_{1}:=\cone\left\{ t_{1}-s_{1},t_{1}-s_{2},t_{2}-s_{1}\right\} .
\end{equation*}
On the other hand, two of the three vectors $t_{1}-s_{1},$ $t_{1}-s_{2}$ and
$t_{2}-s_{1}$ are exterior normals to the parallel sides, and hence they
make an angle of $\pi .$ Those two vectors cannot be $t_{1}-s_{1}$ and $%
t_{1}-s_{2},$ because with such a configuration  $V_{T}(S)$
would be empty. We will consider the two remaining cases.
We start with the case when the two vectors are $t_{1}-s_{1}$ and $%
t_{2}-s_{1},$ say $t_{1}-s_{1}=-\mu \left( t_{2}-s_{1}\right) ,$ with $\mu
>0.$ Since in this case $s_{2}$ does not belong to the straight line
determined by $s_{1},$ $t_{1}$ and $t_{2}$ (because the side of $V_{T}(S)$
contained in the line $\left\langle t_{1}-s_{2},x\right\rangle =\frac{1}{2}%
\left( \left\Vert t_{1}\right\Vert ^{2}-\left\Vert s_{2}\right\Vert
^{2}\right) $ is not parallel to the other two sides)$,$ there is a
circle containing $s_{2},$ $t_{1}$ and $t_{2},$ and we may assume,
w.l.o.g., that the center of this circle is $0_{2}.$ Then $\alpha
_{11}=\alpha _{21}>0$ and $\alpha _{12}=\alpha _{22}=0,$ thus, in view of (%
\ref{ineq}), we have $\gamma _{11}=0=\gamma _{12}$ and then, by (\ref{cone}),%
\begin{equation*}
\gamma _{12}\left( t_{1}-s_{2}\right) =t_{2}-s_{2}=-\left(
t_{1}-s_{1}\right) +t_{1}-s_{2}+t_{2}-s_{1}=t_{1}-s_{2}+\left( 1+\mu \right)
\left( t_{2}-s_{1}\right) ,
\end{equation*}%
which yields%
\begin{equation*}
t_{2}-s_{1}=\frac{\gamma _{12}-1}{1+\mu }\left( t_{1}-s_{2}\right) .
\end{equation*}%
This is impossible because the vectors $t_{1}-s_{1},$ $t_{1}-s_{2}$ and $%
t_{2}-s_{1}$ are not all colinear.
It only remains to consider the case when $t_{1}-s_{2}$ and $t_{2}-s_{1}$
are the vectors that make an angle of $\pi .$ Since $t_{1}-s_{1}$ is not
colinear with these two vectors, we have that the cones $K_{1}$ and%
\begin{equation*}
K_{2}:=\cone\left\{ -\left( t_{1}-s_{1}\right)
,t_{1}-s_{2},t_{2}-s_{1}\right\}
\end{equation*}%
are opposite halfplanes defined by the line containing the vectors $%
t_{1}-s_{2}$ and $t_{2}-s_{1}.$ The fact that $t_{2}-s_{2}\in K_{1}\cap %
\interior K_{2}$ provides a contradiction.
\end{proof}

\bigskip

We note here that it is possible to have an unbounded polygon with three sides
for the case $|S|=3$ and $|T|=4$ if and only if the unbounded sides are
non-parallel. Indeed, in this case we can first build a wedge that defines the
two unbounded sides (see Subsection \ref{Wedges}), and then add an extra site
to define the extra inequality. For the case of unbounded parallel sides, it
is clear that the point $t$ should at the same time lie outside of each of
these parallel sides, which is impossible.

\subsection{Unbounded Polygons with Four Sides}

In the next proposition we shall consider the case of an unbounded
quadrilateral in $\R^{2}$ with two parallel sides, shown in
Fig.~\ref{fig:unbounded-pic3-construction}. \begin{figure}[th]
\centering \includegraphics[width=0.6\textwidth]{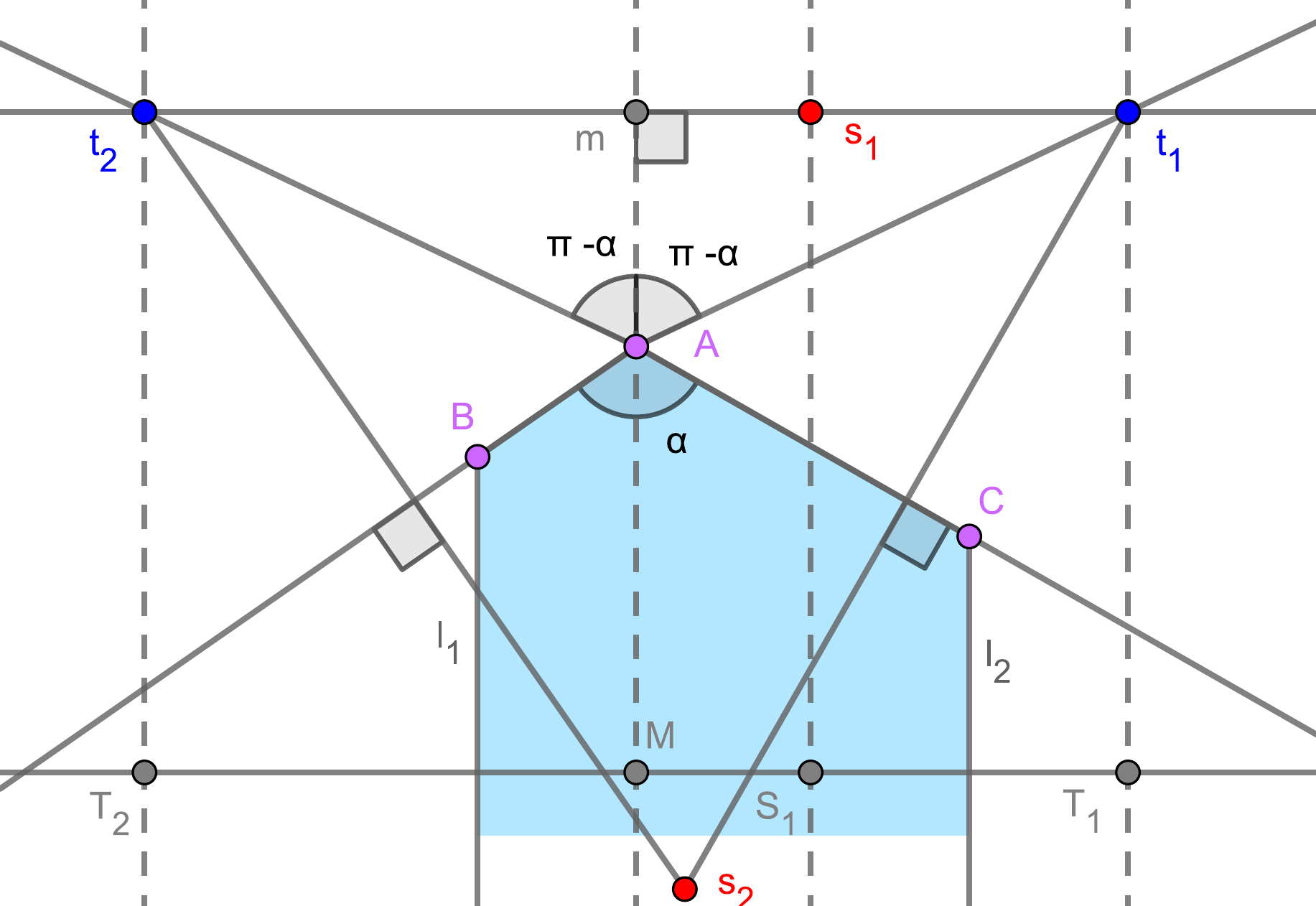}\caption{Unbounded
polygon with parallel sides}%
\label{fig:unbounded-pic3-construction}%
\end{figure}

\begin{proposition}
\label{prop:alg-two} Let $F\subset\R^{2}$ be an unbounded convex
quadrilateral with two parallel sides. Then there exist {$S\subset
T\subset\R^{2}$}, with $|S|=2$ and $|T|=4$, such that $V_{T}(S)=F$.
\end{proposition}

\begin{proof}
In the notation of Fig.~\ref{fig:unbounded-pic3-construction}, where $F$ is
depicted as the unboubed convex polygon with vertices $A,$ $B$ and $C,$ we
have $l_{1}\parallel l_{2}$ and $0<\alpha <\pi ,$ and the point $A$ belongs
to the interior of the parallel band. First, we move, parallel to the line $%
l_{1},$ the point $A$ to an arbitrary point $M$. Through $M$ we draw a line
perpendicular to $l_{1}.$ Take a point $S_{1}$ on the same line, with the
same distance to $l_{2}$ as $M$ is to $l_{1}.$ Now, we consider the
symmetrical point $T_{1}$ and $T_{2}$ of $S_{1}$ with respect to $l_{2}$ and
$l_{1},$ respectively$.$ In this way $M$ is the midpoint of the segment $%
T_{1}T_{2}$. Next we translate the line through the points $T_{1}$ and $T_{2}$
in such a way that, defining $t_{1},$ $t_{2},$ $m$ and $s_{1}$ as the
orthogonal projections of $T_{1},$ $T_{2},$ $M$ and $S_{1}$ onto the
resulting translated line, one has $\sphericalangle $ $t_{1}At_{2}=2\pi
-2\alpha $ and $\sphericalangle $ $mAt_{2}=\pi -\alpha .$ For the last site,
we take $s_{2}$ as the symmetric point of $t_{2}$ with respect to the line
through the points $A$ and $B.$ We shall prove that the line through the
points $t_{1}$, $s_{2}$ is perpendicular to $AC.$ For this purpose, we
observe that the point $A$ is the center of the circle passing through
the points $t_{1}$, $t_{2}$ and $s_{2}.$ This means that $\sphericalangle
t_{2}s_{2}t_{1}=\pi -\alpha .$ If we consider the quadrilateral formed by
the lines $BA,$ $AC,$ $t_{1}s_{2}$ and $t_{2}s_{2},$ we get the desired
fact. At the end, after having found the desired sets $S:=\left\{
s_{1},s_{2}\right\} $ and $T:=\left\{ s_{1},s_{2},t_{1},t_{2}\right\} $, we
would like to mention that this construction is impossible if $\alpha =\pi .$
\end{proof}

\bigskip

We next consider the case of an arbitrary unbounded convex quadrilateral with
nonparallel sides, presented in Fig.~\ref{fig:quadrilateral-general-unbounded}%
. \begin{figure}[th]
\centering \includegraphics[width=0.9\textwidth]{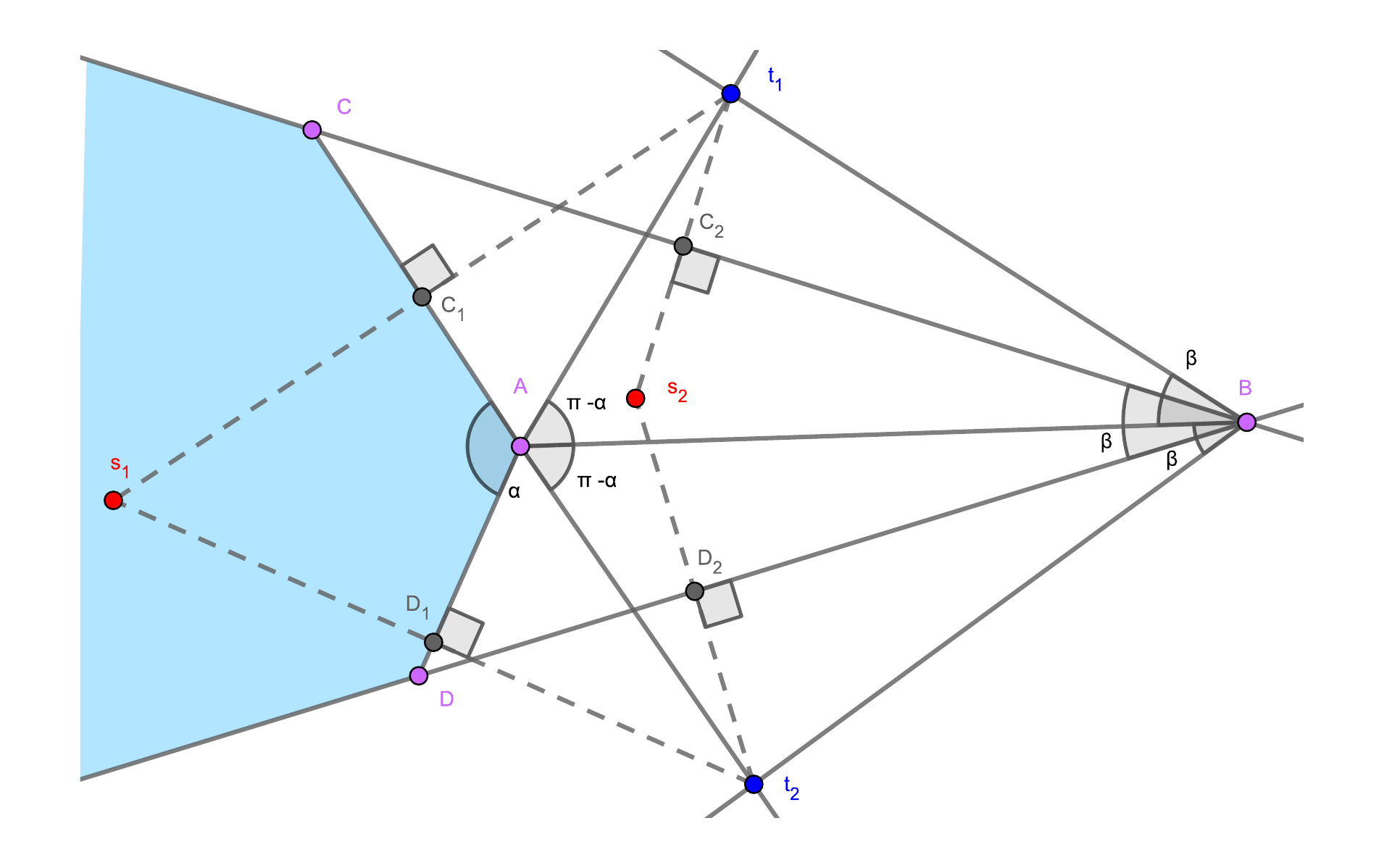}\caption{Configuration
for a general unbounded quadrilateral}%
\label{fig:quadrilateral-general-unbounded}%
\end{figure}

\begin{proposition}
\label{prop:alg-three} Let $F$ be an unbounded convex quadrilateral with no
sides parallel. Then there exist {$S\subset T\subset\R^{2}$}, with
$|S|=2$ and $|T|=4$, such that $V_{T}(S)=F$.
\end{proposition}

\begin{proof}
First, looking at Fig.~\ref{fig:quadrilateral-general-unbounded}, where $F$
is depicted as the shaded region, we obtain the point $B$ as the
intersection of the unbounded sides of the convex quadrilateral with
vertices $CAD$. We denote $\sphericalangle CAD=\alpha $ and $\sphericalangle
CBD=\beta $ and observe that $\pi >\alpha >\beta >0.$ We next draw the lines
that make angles $\pi -\alpha $ with $AB$ at the point $A$ and those that
make angles $\beta $ with the same segment at $B.$ We then get the points $%
t_{1}$ and $t_{2}$ as the two other vertices of the quadrilateral determined
by these four lines. We take $s_{1}$ and $s_{2}$ as the symmetric points of $%
t_{1}$ with respect to $CA$ and $CB,$ respectively. To see that $V_{T}(S)=F$
for $S:=\left\{ s_{1},s_{2}\right\} $ and $T:=\left\{
s_{1},s_{2},t_{1},t_{2}\right\} ,$ the only thing that we have to prove is
that $t_{2}s_{1}\bot AD$ and $t_{2}s_{2}\bot BD.$ The p\`{o}int $B$ is the
center of the circle passing through the points $t_{1},$ $t_{2}\ $and$%
\ s_{2}.$ Therefore, $\sphericalangle t_{1}s_{2}t_{2}=\pi -\beta ,$ which,
looking at the quadrilateral $C_{2}s_{2}D_{2}B,$ implies that $%
t_{2}s_{2}\bot BD.$ The point $A$ is the center of the circle passing
through the points $t_{1},$ $t_{2}$ and$\ s_{1}.$ Therefore, $%
\sphericalangle t_{1}s_{1}t_{2}=\pi -\alpha ,$ which, looking at the
quadrilateral $C_{1}s_{1}D_{1}A,$ implies $t_{2}s_{1}\bot AD.$
\end{proof}

\section{Conclusions}

We have obtained constructive characterizations of properties pertaining to
higher-order Voronoi cells, and applied these characterizations to a case
study of cells of order at least two in the system of four sites.

In particular, the results obtained in the preceding subsections show that for
{$S\subset T\subset\R^{2}$}, with $|S|=2$ and $|T|=4$, the resulting
cell $V_{T}(S)$ can be any of the following sets: the empty set, a singleton,
a non-cyclic bounded convex quadrilateral, a halfplane, an intersection of
parallel halfplanes with opposite normals and a nonempty interior, an angle,
an unbounded polygon with non-parallel sides and just two vertices, or an
unbounded convex quadrilateral. All the remaining possibilities for sets in
{$\R^{2}$ }de ned by four linear inequalities, namely, a singleton, a
one-dimensional set, a triangle, a cyclic quadrilateral, and an unbounded
quadrilateral with parallel sides and just two vertices are unfeasible.

The restrictions on the possible shapes of higher-order cells discovered in
our case study are consistent with the challenges encountered in the design of
numerical methods for constructing higher-order cells (cf. \cite{Sweepline}), where
the key assumption of general position ensures the absence of cyclic
quadrilateral configurations). Thus a natural direction for the future study
is to obtain general structural results on the shapes of higher-order cells,
which will in turn inform the design of algorithms. For instance, we would
like to know whether the convex hull of affinely independent points can be
represented as a higher-order Voronoi cell (generalizing our result on the
impossibility of a triangular cell), what higher-dimensional configurations
produce cells with empty interiors, and what are the possible dimensions of
higher-order Voronoi cells in $\R^{n}$.

\section{Acknowledgments}

We are grateful to the two JOTA referees for their thoughtful and thorough corrections and suggestions, including the proof of Proposition~\ref{prop:impunbthree}. These corrections have greatly improved the quality of our paper. 

The first author was supported by the MINECO of Spain, Grant
MTM2014-59179-C2-2-P, and the Severo Ochoa Programme for Centres of Excellence
in R\&D [SEV-2015-0563]. He is affiliated to MOVE (Markets, Organizations and
Votes in Economics). He thanks The School of Mathematics and Statistics of
UNSW Sydney for sponsoring a visit to Sidney to complete this work.

The second author is grateful to the Australian Research Council for
continuous financial support via grants DE150100240 and DP180100602, which in
particular sponsored a trip to Barcelona that initiated this collaboration.

The third author was partially supported by MINECO of Spain and ERDF of EU,
Grant MTM2014-59179-C2-1-P, and Sistema Nacional de Investigadores, Mexico.


\end{document}